\documentclass[preprint]{imsart}

\RequirePackage{amsthm,amsmath,amsfonts,amssymb}
\RequirePackage[authoryear]{natbib}
\RequirePackage[colorlinks,citecolor=blue,urlcolor=blue,backref=page,backref=page]{hyperref}
\RequirePackage{graphicx}

\usepackage{bigints}
\usepackage{chngcntr}
\usepackage{amssymb}
\usepackage{wasysym}
\usepackage{color}
\usepackage{bm}
\usepackage{extarrows}
\usepackage{pstricks}
\usepackage{makecell,rotating}
\usepackage[labelsep=quad,indention=10pt]{caption}
\usepackage[list=true]{subcaption}
\usepackage{subcaption}
\usepackage{booktabs}
\usepackage{tabularx}
\usepackage{threeparttable}
\usepackage{float}

\pubyear{2026}
\volume{TBA}
\issue{TBA}
\firstpage{1}
\lastpage{1}

\startlocaldefs
\theoremstyle{plain}

\newtheorem{theorem}{Theorem}[section]
\newtheorem{lemma}[theorem]{Lemma}
\newtheorem{corollary}[theorem]{Corollary}
\theoremstyle{definition}
\newtheorem{definition}[theorem]{Definition}

\newtheorem{remark}[theorem]{Remark}
\theoremstyle{remark}

\def\tsc#1{\csdef{#1}{\textsc{\lowercase{#1}}\xspace}}
\tsc{WGM}
\tsc{QE}

\def \R {\mathbb{R}}

\def \N {\mathcal{N}}
\def \E {\mathbb{E}}
\def \Var {\mathbb{V}\textup{ar}}

\def \bbeta {\bm{\eta}}
\def \btheta {\bm{\theta}}
\def \P {\mathbb{P}}
\def \O {\mathcal{O}}
\def \bY {\mathbf{Y}}
\def \G {\mathcal{G}}
\def \X {\mathcal{X}}
\def \F {\mathcal{F}}
\def \bu {\mathbf{u}}
\endlocaldefs

\begin{document}

\begin{frontmatter}
\title{The minimax optimal convergence rate of posterior density in the weighted orthogonal polynomials}
\runtitle{The minimax optimal convergence rate of posterior density}

\begin{aug}
\author[Y]{\fnms{Yiqi}~\snm{Luo}\ead[label=e1]{luoyiqi1999@163.com}}
\and
\author[Y,X,Z]{\fnms{Xue}~\snm{Luo}\ead[label=e2]{xluo@buaa.edu.cn}\orcid{0000-0003-1187-7599}}

\address[Z] {Corresponding author}
\address[Y]{School of Mathematical Sciences, Beihang University (Shahe Campus), Changping district, Beijing, 102206, P. R. China\printead[presep={,\ }]{e1}}

\address[X]{Key Laboratory of Mathematics, Informatics and Behavioral Semantics (LMIB), Haidian district, Beijing, 100191, P. R. China\printead[presep={,\ }]{e2}}

\runauthor{Y. Luo et al.}
\end{aug}

\begin{abstract}
We investigate Bayesian nonparametric density estimation via orthogonal polynomial expansions in weighted Sobolev spaces. A core challenge is establishing minimax optimal posterior convergence rates, especially for densities on unbounded domains without a strictly positive lower bound.
For densities bounded away from zero, we give sufficient conditions under which the framework of \cite{shen2001} applies directly. For densities lacking a positive lower bound, the equivalence between Hellinger and weighted $L_2$-norm distance fails, invalidating the original theory.
We propose a novel shifting method that lifts the true density $g_0$ to a sequence of proxy densities $g_{0,n}$. We prove a modified convergence theorem applicable to these shifted densities, preserving the optimal rate. We also construct a Gaussian sieve prior that achieves the minimax rate $\varepsilon_n=n^{-p/(2p+1)}$ for any integer $p\geq1$.
Numerical results confirm that our estimator approximates the true density well and validates the theoretical convergence rate.
\end{abstract}



\end{frontmatter}

\section{Introduction}

Bayesian nonparametric density estimation provides a flexible framework for modeling complex data without restrictive parametric assumptions. Given independent observations $Y^n=(Y_1,\ldots,Y_n)$ drawn from an unknown density $g_0$, a fundamental theoretical question in Bayesian nonparametric concerns the frequentist properties of posterior distributions induced by infinite-dimensional priors.
In particular, one seeks conditions under which the posterior converges to the true density at the minimax optimal rate \citep{stone1980,tsybakov2009,rate} under suitable assumptions. A general convergence theory has been developed over the past two decades; see, among others, \cite{ghosal2000,ghosal2007,shen2001,vdv2008}. An overview of Bayesian nonparametric density estimation is given in \cite{Peter2014}.

In the seminal works \cite{ghosal2000} and \cite{shen2001} independently developed a Bayesian framework for i.i.d. data that achieves posterior convergence at the minimax optimal rate $\varepsilon_n = n^{-p/(2p+1)}$, where $n$ denotes the sample size and $p$ the regularity of the true model. Both studies provide sufficient conditions in terms of model complexity control and prior concentration. To adapt this general framework to density estimation, one first parametrizes the unknown density $g_0$ and then places a prior distribution on the parameters.

Existing parameterization approaches mainly fall into two classes: kernel-based methods and basis expansion methods. Kernel methods build density models as convex combinations of local kernel functions, usually combined with Dirichlet process priors or other appropriate distributions on the mixing measure and bandwidth, leading to strong local smoothing performance. Representative examples include Gaussian kernels and their multivariate extensions \citep{kruijer2010, scricciolo2014, bio2013}, 
Laplace kernels \citep{gaofn2017},
exponential power kernels for flexible tail behavior \citep{EJOS2011}, and shape-adaptive polynomial kernels \citep{EJOS2017, rousseau2010, xie2020}. Basis expansion methods, by contrast, project the unknown function onto a global linear space by specifying priors on the coefficients of a chosen basis system. Widely used examples include polynomial bases \citep{bernstein, Norets2024}, trigonometric bases \citep{AOS2007}, and spline bases \citep{ghosal2008, shen2015, Yu2022}.

Regardless of the parametrization adopted, within the general framework of \cite{ghosal2000, shen2001}, the minimax optimal posterior convergence depends on two key conditions: controlling model space entropy and bounding the prior probability mass. If the true density $g_0$ lacks a positive lower bound on an unbounded domain, bounding the Kullback-Leibler (KL) divergence becomes difficult, making both conditions highly challenging to satisfy.
 As a result, to ensure minimax optimal convergence rates, existing literature often either directly restricts $g_0$ to a bounded interval \citep{bernstein, Norets2024, rousseau2010, shen2015} or artificially truncates the supports of both the true underlying distribution and the prior to a bounded interval \citep{gaofn2017}. Even when unbounded domains are considered, additional restrictive assumptions are typically imposed, such as requiring $g_0$ to be strictly positive with a uniform positive lower bound on a sufficiently large interval \citep{kruijer2010, scricciolo2014, bio2013}, imposing stringent integral control conditions on relative derivatives \citep{EJOS2017}, or assuming the existence of a new distance larger than the KL divergence but with bounded entropy \citep{Huang2004}.

In this paper, we parametrize the density with respect to the measure $d\P=w(x)dx$  using the orthogonal polynomial bases  $\{q_j(x)\}_{j=0}^\infty$ with weight $w(x)$, where $x\in I$ and $I$ may be a bounded interval or an unbounded domain such as $\R$ or $\R^+$. The minimax optimal convergence rate of the posterior density is established under two scenarios:
\begin{itemize}
    \item[(1)] The true density $g_0$ admits a strictly positive lower bound. Under this assumption, the framework of \cite{shen2001} can be applied directly by verifying the corresponding conditions; see Theorem \ref{th-bound}.
    \item[(2)] The true density $g_0$ has no strictly positive lower bound. We propose a novel approach to circumvent this obstacle by introducing a positive decreasing sequence $a_n\downarrow0$ and lifting $g_0$ to a proxy density $g_{0,n}\geq a_n>0$. We then prove that, with a suitably chosen decay rate of $a_n$, the posterior convergence rate of $g_0$ in the Hellinger distance coincides with that of $g_{0,n}$; see Lemma \ref{lem-3.1}. At first glance, one might expect this to follow readily by applying Shen's result to $g_{0,n}$. However, even with the positive lower bound, Shen's original framework still cannot be applied directly to $g_{0,n}$, because both the KL neighborhood and the model complexity control depend on $a_n$; see the detailed discussions in Section \ref{sec-3}. To resolve this issue, we establish a modified version of Shen’s result in Theorem \ref{th-modified shen}, which ensures that the minimax optimal convergence rate is preserved under a slightly strengthened prior concentration condition on the KL neighborhood. Using this modified result, we obtain the counterpart of Theorem \ref{th-bound} in Theorem \ref{th-prior-general}.
\end{itemize}

Besides, we further investigate the construction of priors that fulfill all three conditions in Theorem \ref{th-prior-general} in Section \ref{sec-4}. This is another closely related topic in this direction. A large body of existing literature focuses on constructing priors over parameters in specific function classes to achieve minimax optimal posterior convergence rates \citep{EJOS2017, kruijer2010, AOS2007}. 

Finally, to better illustrate our theoretical findings, we carry out two sets of numerical experiments.  The first presents density estimates based on orthogonal polynomial expansions, which are compared with those obtained using trigonometric bases in \cite{AOS2007}. The second examines the decreasing behavior of the Hellinger distance between the estimated and true densities as the sample size increases.

The paper is organized as follows. Section~\ref{sec-2} introduces the preliminary assumptions of the weighted orthogonal polynomial space and the model definitions.
Sections~\ref{sec-3.1}-\ref{sec-3.2}  presents the main theorems regarding the minimax optimal posterior convergence rates for densities with and without strictly positive lower bounds, respectively. Section~\ref{sec-4} presents the specific construction of the Gaussian sieve prior. Numerical illustrations are showed in Section \ref{sec-5}.  Finally, Section~\ref{sec-6} concludes with further discussion.

\section{Notations and Preliminaries}\label{sec-2}

Suppose that $\bY^n=(Y_1,\ldots,Y_n)$ are i.i.d.\ observations drawn from an unknown distribution on the measurable space $(\X,\F)$. We assume that the unknown distribution admits a density $g_0$ with respect to the reference measure $d\P=w(x)\,dx$ on a domain $I$, where $w(x)$ is the weight function associated with an orthogonal polynomial system $\{q_j(x)\}_{j\geq0}$. To avoid ambiguity, we assume that the polynomial of degree zero is $q_0(x)\equiv 1$. Throughout this paper, $I$ is not restricted to a bounded interval and may also be an unbounded one such as $\R$ or $\R^+$.

Fix an integer $p\geq1$, assume that $g_0\in W_{2,w}^p(I)$, the weighted Sobolev space equipped with the norm
\begin{equation}\label{eqn-2.2}
  \|g\|_{W_{2,w}^p}^2
  =\sum_{l=0}^p\|g^{(l)}\|_{L_{2,w}}^2
  =\sum_{l=0}^p\int_I |g^{(l)}(x)|^2 w(x)\,dx
  <\infty,
\end{equation}
where $g^{(l)}$ denotes the $l$-th weak derivative of $g$. The density to be estimated in this paper belongs to the function class
\begin{equation}\label{eqn-G}
  \G
  :=\left\{ g\in W_{2,w}^p(I)\left|\, g\geq0 \ \textup{and}\ \int_I g(x) w(x)dx=1 \right.\right\}.
\end{equation}

For any $g_0\in\G$, it admits a unique expansion in the orthogonal polynomial basis $\{q_j\}_{j\geq0}$, i.e.
\begin{equation}\label{1.1}
  g_0(x)=\sum_{j=0}^{\infty}\theta_j q_j(x),
\end{equation}
where the parameters
$ \theta_j=\frac{1}{\gamma_j}\int_I g_0(x)q_j(x)w(x)dx$, $j=0,1,\cdots$, are determined by the orthogonality relation
\begin{equation}\label{p}
  \int_I q_i(x)q_j(x)w(x)\,dx=\delta_{ij}\gamma_j,
  \qquad i,j=0,1,\cdots,
\end{equation}
with $\gamma_j$ being the normalization constant and $\delta_{ij}$ is the Kronecker delta. We also denote $g_0$ by $g(x|\btheta)$ to emphasize the dependence on the parameter $\btheta=(\theta_0,\theta_1,\cdots)$. The corresponding coefficient space of $\G$ is denoted as $\Omega^\G$. Due to the one-to-one correspondence, we shall not distinguish the use of those two notations in the sequel. 

Next, we specify a subset of the coefficient space $\Omega\subset\Omega^\G$.
\begin{lemma}\label{lem-1}
For any 
\begin{equation}\label{eqn-2.5}
    g(x|\bbeta)=\sum_{j=0}^\infty\eta_jq_j(x)\in\G,
\end{equation}
if 
\begin{align}\label{def-Omega}\notag
  \bbeta\in\Omega
  :=&\left\{ \bbeta=(\eta_0,\eta_1,\cdots)\left|\, \sum_{j=0}^\infty\eta_jq_j(x)\geq0,\ \sum_{j=0}^\infty\eta_j\int_Iq_j(x)w(x)dx=1\right.\right.\\
  & \phantom{aaaaaaaaaaaaaaaa}\left.
 \textup{and}\ \sum_{j=0}^{\infty}\eta_j^2\tilde\gamma_j<\infty \right\},
\end{align}
where $\tilde\gamma_j$ is any constant greater than or equal to 
\begin{equation}\label{eqn-2.3}
   \max\left\{\max_{0\leq l\leq \min(j,p)}\sum_{i=0}^{j-1} (a_{ij}^{(l)})^4\gamma_i,j^4\max_{0\leq i\leq j-1}\gamma_i\right\},
\end{equation}
for $j\geq1$, with $\displaystyle a_{ij}^{(l)}=\frac{1}{\gamma_i}\int_I q_j^{(l)}(x)q_i(x)w(x)\,dx$ being the coefficients in the expansion
\begin{equation}\label{eqn-2.4}
    q_j^{(l)}(x)=\sum_{i=0}^{j-1}a_{ij}^{(l)}q_i(x),
    \end{equation}
for $0\leq l\leq\min(j,p)$.
\end{lemma}

The proof is in Section A of the Supplementary Material \cite{LL:26}. In Section \ref{sec-4} we shall choose $\tilde\gamma_j\geq j^{7p} \gamma_j$, $j\geq1$, to construct the prior.

\begin{remark}\label{remark-tildegamma}
 In the existing literature, a particular basis is usually chosen first, and $\tilde\gamma_j$ is then specified to enforce the Sobolev-type constraint. For instance, the trigonometric basis employed in \cite{shen2001,AOS2007} uses $\tilde\gamma_j=\O(j^{2p})$. By contrast, if the underlying orthogonal polynomial system is not specified in advance, $\tilde\gamma_j$ can vary considerably across different polynomial families. As illustrated by examples in Section B of the Supplementary Material \cite{LL:26}, $\tilde\gamma_j\gtrsim j^{4p+1}$ for Legendre polynomials, and 
 $\tilde\gamma_j\gtrsim j^{4p}\gamma_j \gtrsim j^{4p} 2^j j!$ for Hermite polynomials, when $j\gg1$. Nevertheless, Lemma \ref{lem-1} establishes that the choice given in \eqref{eqn-2.3} ensures the embedding property required for our analysis.
\end{remark}

The problem of interest is to construct Bayesian density estimators that achieve the minimax optimal convergence rate under Hellinger loss. It is well known that for $g_0\in W_{2,w}^p(I)$, the minimax rate is given by $\varepsilon_n=n^{-p/(2p+1)}$; see \cite{stone1980,tsybakov2009,rate}. In this paper, we provide affirmative answers to the following questions. Does there exist a function $g\in\G$ of the form \eqref{eqn-2.5}, together with an appropriate class of prior distributions on  associated with $\{q_j\}_{j\geq0}$, such that the posterior converges to the true density  as $n\to\infty$? If so, what is the corresponding convergence rate, and can the minimax rate $\varepsilon_n$ be attained for any $p\geq1$? To prepare for our analysis, we first recall several key definitions.

\begin{definition}[posterior convergence rate]\label{def-convergence rate}
We call $\varepsilon_n>0$ a convergence rate if the posterior probability
\begin{equation}\label{prosterior}
  \pi(A_n^c|\bY^n) =
  \frac{\displaystyle\int_{A_n^c}\prod_{i=1}^n \frac{g(Y_i|\bbeta)}{g(Y_i|\btheta)}\,d\pi(\bbeta)}
  {\displaystyle\int_{\Omega^\G}\prod_{i=1}^n \frac{g(Y_i|\bbeta)}{g(Y_i|\btheta)}d\pi(\bbeta)}\to 0,
\end{equation}
as $n\to\infty$, where $\pi$ is the prior distribution on $\bbeta\in\Omega^\G$ and
\begin{equation}\label{An-c}
  A_n^c
  :=\left\{ g(\cdot|\bbeta)\in\G \left|\, d\bigl(g(\cdot|\btheta),g(\cdot|\bbeta)\bigr)>K\varepsilon_n \right.\right\}
  =\left\{\bbeta\in\Omega^\G\left|\,d\left(\btheta,\bbeta\right)>K\varepsilon_n \right.\right\}
\end{equation}
for some constant $K>0$ and a distance $d(\cdot,\cdot)$. 
\end{definition}

Throughout this paper, we take $d(\cdot,\cdot)$ in Definition~\ref{def-convergence rate} to be the Hellinger distance $d_H(\cdot,\cdot)$ defined in Definition \ref{def-2.3} below. We also define two related divergences used to characterize Kullback-Leibler (KL) neighborhoods, which is used to quantify the difference between the estimated and true density; see \cite{shen2001,ghosal2000}.
\begin{definition}[Hellinger distance and KL neighborhood]\label{def-2.3}
For any density functions $g_1$ and $g_2$ supported on $I$ under the measure $\P$, the Hellinger distance is
\begin{equation}\label{hellinger distance}
  d_H^2(g_1,g_2)
  :=\int_I\bigl(\sqrt{g_1(x)}-\sqrt{g_2(x)}\bigr)^2 w(x)\,dx.
\end{equation}
For any $t>0$, the KL neighborhood $S(t)$ of the true density $g_0$ is 
\begin{equation}\label{S}
  S(t)
  :=\left\{ g\in\G \left|\, \max\left(D(g_0,g),V(g_0,g)\right)\leq t \right.\right\},
\end{equation}
where 
\begin{equation}\label{KL distance}
  D(g_1,g_2)
  :=\E_{g_1}\bigl(\log g_1-\log g_2\bigr)
  =\int_I g_1(x)\log\frac{g_1(x)}{g_2(x)}\,w(x)\,dx,
\end{equation}
is the KL divergence and
\begin{equation}\label{variance distance}
  V(g_1,g_2)
  := \Var_{g_1}\bigl(\log g_1-\log g_2\bigr)
  = \E_{g_1}\Bigl[\bigl(\log g_1-\log g_2\bigr)^2\Bigr]
  -\bigl(D(g_1,g_2)\bigr)^2
\end{equation}
is the variance of the log-likelihood ratio.
\end{definition}

To gauge the complexity of the model class $\G$, the bracketing entropy with respect to a metric $d(\cdot,\cdot)$ is introduced.
\begin{definition}[bracketing entropy, Section~2.7, \cite{b20}]\label{bracket function}
A finite set of pairs of functions $\{(l_i,u_i)\}_{i=1}^M$ is a $d(\cdot,\cdot)$ $\varepsilon$-bracketing of $\G$ if for any $g\in\G$ there exists a pair $(l_j,u_j)$ such that $l_j\leq g\leq u_j$ almost everywhere and $d(l_j,u_j)\leq \varepsilon$. The smallest number of such brackets is the bracketing number, denoted by $N_{[~]}(\varepsilon,\G,d)$. The logarithm of the bracketing number is the bracketing entropy, denoted by $H_{[~]}(\varepsilon,\G,d)$, that is,
\begin{equation}\label{eqn-bracketing-entropy}
  H_{[~]}(\varepsilon,\G,d)
  := \log N_{[~]}(\varepsilon,\G,d).
\end{equation}
\end{definition}

With all necessary notations and definitions in place, we state Shen’s convergence rate framework in Theorem \ref{th-shen} for the convenience of readers. In general, \cite{shen2001} established sufficient conditions for achieving the minimax optimal convergence rate in Bayesian models with i.i.d. observations. These conditions are based on three key components: (i) entropy control of the model class, (ii) prior mass in KL neighborhood of the true distribution, and (iii) the existence of a high-probability set under the prior. In particular, they proposed a general framework demonstrating posterior convergence, provided that the three bounds \eqref{th-shen-1}--\eqref{th-shen-2} in Theorem \ref{th-shen} below can be verified.

\begin{theorem}[Theorem~4, \cite{shen2001}]\label{th-shen}
Suppose that there exists a sequence $r_n\to0$, constants $C_1,C_2>0$, and a sequence of subsets $\Omega^\G_1\subset\Omega^\G_2\subset\cdots\subset\Omega^\G$ such that
\begin{equation}\label{th-shen-1}
  \int_{r_n^2/2^8}^{\sqrt{2}\,r_n}
  H_{[~]}^{1/2}\left(u/C_1,\Omega^\G_n,d_H\right)du
  \lesssim \sqrt{n}\,r_n^2,
\end{equation}
and there exists a sequence $t_n\to0$ such that
\begin{equation}\label{th-shen-3}
  \pi\left(S\left(\frac{C_2 t_n}{16}\right)\right)
  \gtrsim \exp\left(-\frac{C_2 n t_n}{8}\right).
\end{equation}
Let $\varepsilon_n=\max\{r_n,t_n^{1/2}\}$, and assume further that
\begin{equation}\label{th-shen-2}
  \pi\left({\Omega^\G_n}^c\right)\leq \exp\left(-C_2 n \varepsilon_n^2\right).
\end{equation}
Then, for $K>0$ sufficiently large, if $n\varepsilon_n^2\to\infty$, it holds that
\begin{equation}\label{eqn-convergence}
  \pi(A_n^c|\bY^n)
  \lesssim
  \exp\!\left(-\frac{nK^2\varepsilon_n^2}{2}\right)
  +\exp\!\left(-\frac{C_2 n\varepsilon_n^2}{4}\right)
  \to0,
\end{equation}
except on a set of probability tending to $0$, where $A_n^c$ is in \eqref{An-c} with $d_H$. Here $a_n\lesssim b_n$ (resp.\ $a_n\gtrsim b_n$) means that there exists a constant $C>0$ such that $a_n\leq Cb_n$ (resp.\ $a_n\geq Cb_n$) for all large $n$.
\end{theorem}

\begin{remark}\label{remark-notation-shen}
As noted earlier, we do not distinguish between the parameter space  and the function space $\Omega^\G$, owing to their one-to-one correspondence. For instance, $H_{[~]}(\cdot,\Omega^\G_n,\cdot)$ coincides with $H_{[~]}(\cdot,\mathcal{G}_n,\cdot)$ in \eqref{th-shen-1}, where $\G_n$ denotes the function space corresponding to $\Omega_n^\G$. The parameters $\btheta,\bbeta\in\Omega^\G$ correspond to $g_0,g\in\G$, and we identify $d_H(\btheta,\bbeta)$ with $d_H(g_0,g)$, and so forth.
\end{remark}

\begin{remark}\label{remark-compact-noncompact}
(1) In general, $\Omega_n^\G$ in Theorem \ref{th-shen} can be any sequence of subsets. One may choose $\Omega_n^\G=\Omega_n\cap\Omega^\G$, where
\begin{equation}\label{eqn-O_M}
  \Omega_n
  :=\left\{\bbeta\in\Omega\left|\, \sum_{j=0}^{\infty}\eta_j^2\tilde\gamma_j\leq n^{\frac{1}{2p+1}}\right.\right\},
\end{equation}
if \eqref{th-shen-1}--\eqref{th-shen-2} are satisfied. This is exactly what we show in Theorem \ref{th-bound}, when $g_0\in\G$ admits a positive lower bound on $I$. 

(2) Theorem \ref{th-shen} also covers the compact case $\Omega^\G=\Omega_M$, for some constant $M>0$ independent of $n$, in which condition \eqref{th-shen-2} holds automatically since $\pi\left((\Omega_n\cap\Omega_M)^c\right)=0$ for all $n>M$, where $n$ is the sample size.
\end{remark}


\begin{remark}\label{remark-ghosal}

\cite{ghosal2000} proposed three analogous conditions---model complexity, prior mass, and sieve space---and also established minimax optimal posterior convergence rates. A key technical distinction lies in the metric used to quantify model complexity: \cite{ghosal2000} mainly controls the $\varepsilon$-packing number $D(\varepsilon, \mathcal{P}_n, d)$ or the $\varepsilon$-covering number $N(\varepsilon, \mathcal{P}_n, d)$ via an $\varepsilon$-net argument, whereas Theorem~\ref{th-shen} employs the bracketing number $N_{[\ ]}(\varepsilon, \mathcal{P}_n, d)$. As observed in \cite{ghosal2000}, since any bracket of size $\varepsilon$ is contained in a ball of radius $\varepsilon/2$, the covering number is naturally bounded above by the bracketing number, i.e., $N(\frac{\varepsilon}{2}, \mathcal{P}_n, d) \le N_{[\ ]}(\varepsilon, \mathcal{P}_n, d)$. While bracketing entropy is generally larger than packing entropy, the two quantities yield equivalent results up to a constant factor in many classical infinite-dimensional settings. Accordingly, the theoretical mechanisms underlying the two approaches are essentially analogous. In view of this equivalence and for analytical convenience in dealing with our particular model formulation, we adopt the framework of Theorem~\ref{th-shen} for the rest of this paper.
\end{remark}

\section{Convergence Theorems}\label{sec-3}

In this section, we establish posterior convergence of $g\in\G$ to the true density $g_0\in\G$ with the minimax optimal rate $\varepsilon_n$, as defined in Definition \ref{def-convergence rate}. A strictly positive lower bound for the target density $g_0\in\G$ is technically critical for verifying the conditions of Theorem \ref{th-shen}, as it facilitates the connection between KL-type quantities and the $L_{2,w}$-distance. In Section \ref{sec-3.1}, under the assumption that there exists some $a>0$ such that
\begin{equation}\label{eqn-Ga}
 g_0\in \G_a
  :=\left\{ g\in W_{2,w}^p(I)\left|\, g(x)\geq a>0 \ \textup{and} \ \int_I g(x)w(x)dx=1\right.\right\},
\end{equation}
we directly verify the conditions of Theorem \ref{th-shen} to establish Theorem \ref{th-bound}.

However, densities over unbounded domains do not necessarily admit a uniform positive lower bound. A typical example is the Gaussian density on $\R$ with respect to the Lebesgue measure, which decays to zero as $|x|\to\infty$. While the overall proof strategy remains analogous, it requires more refined analysis. In Section \ref{sec-3.2}, we introduce a decreasing sequence $a_n\downarrow0$ and shift the true density to $g_{0,n}\in\G_{a_n}$ via \eqref{tilde}, ensuring that $g_{0,n}\geq a_n>0$. Nevertheless, Theorem \ref{th-shen} still cannot be applied directly, as the constant $C_2$ depends on $a_n$ (see \eqref{eqn-3.11}), which precludes the uniform validity of condition \eqref{th-shen-3} over all $n$. We therefore establish a modified version of Shen's theorem tailored to $g_{0,n}\in\G_{a_n}$ (see Theorem \ref{th-modified shen}), which yields $\pi(A_{n,\btheta_n}^c|\bY^n)\to0$, where
\begin{equation}\label{def-An-theta-n}
  A_{n,\btheta_n}^c
  :=\left\{g\in\G_{a_n}\left|\, d_H(g,g_{0,n})>K\varepsilon_n\right.\right\},
\end{equation}
and $\btheta_n$ denotes the parameter corresponding to $g_{0,n}$. Finally, we transfer this result to the original true density $g_0$ by showing that $d_H(g,g_0)\lesssim d_H(g_n,g_{0,n})$ on $A_n^c$ (see Lemma \ref{lem-3.1}). The counterpart of Theorem \ref{th-bound} is presented in Theorem \ref{th-prior-general}.

\subsection{Convergence theorem for the strictly positive lower bound case}\label{sec-3.1}

In this subsection, we assume that $g_0(x)\in\G_a$ for $x\in I$, so that Theorem~\ref{th-shen} can be applied directly. We state our first main result below.

\begin{theorem}\label{th-bound}
Let $g_0$ and $g\in\G_a$, for some constant $a>0$. The corresponding coefficients are $\btheta$ and $\bbeta\in\Omega^{\G_a}$, where $\Omega^{\G_a}$ is the coefficient space associated with $\G_a$. Let $\pi$ be a prior on $\bbeta=(\eta_0,\eta_1,\cdots)$. For any $\varepsilon_n\geq n^{-p/(2p+1)}$ and any integer $p\geq1$, suppose that there exists a constant $C_3>0$, independent of $n$, such that
\begin{equation}\label{eqn-3.1}
  \pi\left(
  \left\{\bbeta\in\Omega^{\G_a}\left|\, \sum_{j=0}^\infty(\eta_j-\theta_j)^2\gamma_j \leq aC_3\varepsilon_n^2\right. \right\}
  \right)
  \gtrsim \exp\left(-2C_3n\varepsilon_n^2\right),
\end{equation}
and
\begin{equation}\label{th2-Omega}
  \pi\left(\left\{\bbeta\in\Omega^{\G_a}\left|\, \sum_{j=0}^\infty \eta_j^2\tilde\gamma_j>n^{\frac{1}{2p+1}}\right.\right\}\right)
  \lesssim \exp\left(-C_3n\varepsilon_n^2\right),
\end{equation}
where $\tilde\gamma_j$ is defined in \eqref{eqn-2.3}. If $n\varepsilon_n^2\to\infty$, then for sufficiently large $K>0$,
\begin{equation}\label{eqn-3.2}
  \pi(A_{a,n}^c|\bY^n)
  \lesssim
  \exp\!\left(-\frac{nK^2\varepsilon_n^2}{2}\right)
  +\exp\!\left(-4C_3n\varepsilon_n^2\right)
  \to0,
\end{equation}
except on a set of probability tending to $0$, where
$A_{a,n}^c
  :=\left\{\bbeta\in\Omega^{\G_a} \left|\, d_H(g,g_0)\geq K\varepsilon_n\right.\right\}$.
\end{theorem}

\begin{proof}
We apply Theorem~\ref{th-shen} with $\Omega^\G=\Omega^{\G_a}$, $\Omega_{a,n}=\Omega^{\G_a}\cap\Omega_n$, with the sieve $\Omega_n$ defined in \eqref{eqn-O_M}, $\varepsilon_n=r_n=t_n^{1/2}$, $C_1=1$, and $C_2=16C_3$. It suffices to verify \eqref{th-shen-1}--\eqref{th-shen-2}.

\begin{enumerate}
\item[(1)] The sieve-mass condition \eqref{th-shen-2} follows from \eqref{th2-Omega}. Indeed,
\begin{align}\label{eqn-3.10}\notag
  \pi(\Omega_{a,n}^c)
  &=\pi\left(\Omega^{\G_a}\backslash\Omega_n\right)
  \overset{\eqref{eqn-O_M}}{=}
   \pi\left(\left\{\bbeta\in\Omega^{\G_a} \left|\, \sum_{j=0}^\infty\eta_j^2\tilde\gamma_j>n^{\frac1{2p+1}}\right.\right\}\right)\\
   &\overset{\eqref{th2-Omega}}{\lesssim}
   \exp\left(-    C_3n\varepsilon_n^2\right).
\end{align}

\item[(2)] To verify \eqref{th-shen-3}, for $t>0$, let us define
\begin{equation}\label{Soverline}
  \bar S_a(t)
  :=\left\{\bbeta\in\Omega^{\G_a}\left|\, \sum_{j=0}^\infty(\eta_j-\theta_j)^2\gamma_j \leq at\right.\right\}.
\end{equation}
We claim that $\bar S_a(t)\subset S_a(t)$, where \begin{equation}\label{Sa-def}
  S_a(t)
  :=\left\{g\in\G_a \left|\, \max\{D(g_0,g),V(g_0,g)\}\leq t\right.\right\}.
\end{equation}
In fact, for any $\bbeta\in\bar S_a(t)$, the orthogonality relation \eqref{p} yields
\begin{equation}\label{eqn-3.5}
  \|g-g_0\|_{L_{2,w}}^2
  =\int_I (g(x)-g_0(x))^2 w(x)dx
  =\sum_{j=0}^\infty(\eta_j-\theta_j)^2\gamma_j.
\end{equation}
Using $\log x\leq x-1$ for $x>0$ and $g,g_0\in\G_a$, we have
\begin{align}\label{D}
  D(g_0,g)
  =&\int_I g_0(x)\log\left(\frac{g_0(x)}{g(x)}\right)w(x)dx
  \leq \int_I g_0(x)\left(\frac{g_0(x)}{g(x)}-1\right)w(x)dx\\\notag
  =& \int_I \frac{(g_0(x)-g(x))^2}{g(x)}w(x)dx
  \leq \frac{1}{a}\|g-g_0\|_{L_{2,w}}^2
 =\frac{\sum_{j=0}^\infty(\eta_j-\theta_j)^2\gamma_j}{a}
  \overset{\eqref{Sa-def}}{\leq} t.
\end{align}
Similarly,
\begin{align}\label{V}
  V(g_0,g)
  &=\Var_{g_0}(\log g_0-\log g)
   \leq \int_I g_0(x)\left(\log\frac{g_0(x)}{g(x)}\right)^2 w(x)dx \notag\\
  &\leq \int_I g_0(x)\left(\frac{g(x)-g_0(x)}{g_0(x)}\right)^2 w(x)dx
  \leq \frac{1}{a}\|g-g_0\|_{L_{2,w}}^2
  \overset{\eqref{Sa-def}}\leq t.
\end{align}
Thus $\bbeta\in S_a(t)$, proving $\bar S_a(t)\subset S_a(t)$. Taking $t=C_3\varepsilon_n^2$, it yields that
\begin{align}\label{eqn-3.11}
  \pi\left(S_a(C_3\varepsilon_n^2)\right)
  &\geq \pi\left(\bar S_a(C_3\varepsilon_n^2)\right)
   \overset{\eqref{eqn-3.1}}{\gtrsim}
    \exp\left(-2C_3n\varepsilon_n^2\right),
\end{align}
which is \eqref{th-shen-3} with $t_n=C_3\varepsilon_n^2$ and $C_2=16C_3$.

\item[(3)] It remains to verify \eqref{th-shen-1}. For $0<\varepsilon\ll1$, we use
\begin{equation}\label{proof-th1-H}
  H_{[~]}\left(\varepsilon,\Omega_{a,n},d_H\right)\leq H_{[~]}\left(\varepsilon,\Omega_{a,n},L_{2,w}\right)
  \leq H_{[~]}\left(\varepsilon,W_{2,w}^p,L_{2,w}\right)
  \lesssim \varepsilon^{-1/p},
\end{equation}
where the second inequality is due to $\Omega_{a,n}\subset\Omega^{\G_a}$ whose corresponding function space $\G_{a}\subset W_{2,w}^p(I)$ by \eqref{eqn-Ga}; the third inequality follows from Theorem 5.2, \cite{b23}, which states that for any integer $p\geq1$ and sufficiently small $\varepsilon>0$, $ H_{[~]}\left(\varepsilon,W_{2,w}^p,L_{2,w}\right)\lesssim \varepsilon^{-1/p}$.

To show the first inequality in \eqref{proof-th1-H}, we need to compare the Hellinger 
distance $d_H$ and the $L_{2,w}$-norm. Let ${\left\{(l_i,u_i)\right\}}_{i=1}^M$ are pairs of $\varepsilon$-bracket function of $\mathcal{G}_{a,n}$ (the function space corresponding to $\Omega_{a,n}$) with respect to $d_H$. Therefore, $\sqrt{ u_i(x)}+\sqrt{ l_i(x)}\geq\sqrt{a}$, for all $i=1,\cdots,M $. Notice that for each $i=1,\cdots,M$, one has
\begin{align} \label{l2-h}\notag
    d_H^2(u_i,l_i)&=\int {\left( \sqrt{u_i}-\sqrt{l_i}  \right)}^2wdx \leq \int \frac{{\left(u_i-l_i \right)}^2}{{\left( \sqrt{u_i}+\sqrt{l_i}  \right)}^2}wdx\\
     & \leq\frac{1}{a} \int {\left(u_i-l_i \right)}^2wdx
        =\frac1{a}\Vert u_i-l_i \Vert_{L_{2,w}}^2.
\end{align}
If ${\left \{ (l_i,u_i) \right \}}_{i=1}^M$ as assumed before are pairs of $\varepsilon$-bracket of $\Omega_{a,n}$ with respect to $d_H$, then $d_H(l_i,u_i)\leq \varepsilon$, for all $i=1,2,\cdots,M$. From \eqref{l2-h}, one may have $\Vert u_i-l_i \Vert _{L_{2,w}} > \varepsilon$, for some $i=1,\cdots,M$. If constructing the pairs of the $\varepsilon$-bracket of $\Omega_{a,n}$ with respect to $L_{2,w}$, then more pairs than the original ${\{ (l_i,u_i) \}}_{i=1}^M$ may be needed. Thus, $N_{[~]}(\varepsilon,\Omega_{a,n},d_H) \leq N_{[~]}(\varepsilon,\Omega_{a,n},L_{2,w})$, so does $H_{[~]}$.

Integrating \eqref{proof-th1-H} yields
\begin{align}\label{proof-th1-int}
\int_{\varepsilon_n^2/2^8}^{\sqrt{2}\varepsilon_n}
  H_{[~]}^{1/2}\!\left(\varepsilon,\Omega_{a,n},d_H\right)\,d\varepsilon
  &\lesssim
  \int_{\varepsilon_n^2/2^8}^{\sqrt{2}\varepsilon_n}
  \varepsilon^{-\frac{1}{2p}}\,d\varepsilon
  \lesssim \varepsilon_n^{\frac{2p-1}{2p}}
  \leq \sqrt{n}\,\varepsilon_n^2,
\end{align}
since $n^{-p/(2p+1)}\leq\varepsilon_n\to0$.
\end{enumerate}
The convergence result \eqref{eqn-3.2} follows immediately from Theorem~\ref{th-shen} with \eqref{eqn-3.10}, \eqref{eqn-3.11} and \eqref{proof-th1-int}.
\end{proof}

\subsection{Convergence theorem for the case without strictly positive lower bound}\label{sec-3.2}

If the density functions does not have a strictly positive lower bound, we 
shift densities by a sequence $a_n\downarrow0$. For $0\leq a_n<1/\gamma_0$, define the shifted density
\begin{equation}\label{tilde}
  g_n(x):=a_n+g(x)\bigl(1-a_n\gamma_0\bigr),
\end{equation}
where $\displaystyle\gamma_0=\int_I q_0^2(x)w(x)\,dx=\int_I w(x)\,dx$ (recall $q_0\equiv1$). It is straightforward to verify that $g_n\in\G_{a_n}$, where
\begin{equation}\label{def-Gan}
  \G_{a_n}
  :=\left\{g\in W_{2,w}^p(I)\left|\, g\geq a_n>0 \ \textup{and} \ \int_I g(x)w(x)dx=1 \right.\right\},
\end{equation}
as in \eqref{eqn-Ga}, and $\G_{a_n}\subset \G_{a_{n+1}}\subset\cdots\subset\G$. The corresponding coefficient of $g_n$, denoted as $\bbeta_n=(\eta_{0,n},\eta_{1,n},\ldots)\in\Omega^{\G_{a_n}}$ is associated with $\bbeta=(\eta_0,\eta_1,\cdots)$, the coefficient of $g$, by 
\begin{equation}\label{eqn-2.1}
  \eta_{j,n}
  :=\left\{
  \begin{aligned}
    &a_n+\eta_0(1-a_n\gamma_0), & j=0\\
    &\eta_j(1-a_n\gamma_0), & j\geq1
  \end{aligned}\right.  ,
\end{equation}
by \eqref{tilde}. Here, if Theorem \ref{th-bound} applies to $g_n,g_{0,n}\in\G_{a_n}$---though in practice the modified version of Theorem \ref{th-shen} is needed to ensure minimax optimal convergence; see Theorem \ref{th-modified shen}---we obtain $\pi\left(A_{n,\btheta_n}^c|\bY^n\right)\to0$, as $n\to\infty$, where $A_{n,\btheta_n}^c$ is defined in \eqref{def-An-theta-n}.
Lemma \ref{lem-3.1} below establishes that, if $a_n \lesssim \varepsilon_n^4$, then the Hellinger distances associated with the original pair $(g,g_0)$ and the shifted pair $(g_n,g_{0,n})$ are comparable. This implies
$\pi\left(A_n^c|\bY^n\right)\leq \pi\left(A_{n,\btheta_n}^c|\bY^n\right)$,
which in turn yields the desired convergence $\pi(A_n^c|\bY^n)\to0$. 

\begin{lemma}\label{lem-3.1}
Let $g_n$ and $g_{0,n}$ be the shifted densities from $g$ and $g_0$ via \eqref{tilde}. For sufficiently large $K>0$, one has
\begin{equation}\label{dH-transfer}
  d_H(g,g_0)\lesssim d_H(g_n,g_{0,n}),
\end{equation}
on $g\in A_n^c$ in \eqref{An-c}, provided that
\begin{equation}\label{an-cond}
  0<a_n\leq \frac{K^4\varepsilon_n^4}{(16+K^4\varepsilon_n^4)\gamma_0}.
\end{equation}
\end{lemma}

\begin{proof}
Since $\displaystyle d_H^2(g,g_0)=2-2\int_I\sqrt{gg_0}\,w(x)\,dx$, it suffices to show that there exists $\alpha>0$ such that
\begin{equation}\label{eqn-3.13}
 \alpha
  \leq
  \frac{1-\int_I\sqrt{g_ng_{0,n}}\,w(x)\,dx}{1-\int_I\sqrt{gg_0}\,w(x)\,dx}.
\end{equation}

We first bound the numerator. Using \eqref{tilde} and $\sqrt{u_1+u_2+u_3}\leq \sqrt{u_1}+\sqrt{u_2}+\sqrt{u_3}$, for any $u_1,u_2,u_3\geq0$, one has
\begin{align}\label{eqn-3.14}\notag
  &1-\int_I\sqrt{g_ng_{0,n}}w(x)dx\\\notag
  \geq& (1-a_n\gamma_0)\left(1-\int_I\sqrt{gg_0}\,w(x)\,dx\right)
 -\sqrt{a_n(1-a_n\gamma_0)}\int_I\sqrt{g+g_0}w(x)dx \\
  \geq& (1-a_n\gamma_0)\left(1-\int_I\sqrt{gg_0}w(x)dx\right)
  -2\sqrt{a_n\gamma_0(1-a_n\gamma_0)}.
\end{align}
The last inequality uses $\displaystyle \int_I\sqrt{g}w(x)dx\leq \left(\int_Igw(x)dx\right)^{1/2}\left(\int_Iw(x)dx\right)^{1/2}=\sqrt{\gamma_0}$, by H\"older's inequality and the fact that $\displaystyle\int_I gw(x)dx=1$, similarly for $g_0$.

Thus, from \eqref{eqn-3.14}, one has
\begin{equation}\label{eqn-3.16}
  \frac{1-\int_I\sqrt{g_ng_{0,n}}\,w(x)\,dx}{1-\int_I\sqrt{gg_0}\,w(x)\,dx}
  \geq
  (1-a_n\gamma_0)
  -\frac{4\sqrt{a_n\gamma_0(1-a_n\gamma_0)}}{K^2\varepsilon_n^2}, 
\end{equation}
on $A_n^c$, since the denominator on the right-hand side of \eqref{eqn-3.13} has a lower bound $d^2_H(g,g_0)\geq K^2\varepsilon_n^2$. 

It is straight-forward to verify that the right-hand side of \eqref{eqn-3.16} admits a strictly positive lower bound independent of $n$ whenever \eqref{an-cond} holds. This implies \eqref{eqn-3.13} for some $\alpha>0$, proving \eqref{dH-transfer}.
\end{proof}

Here, we address the impossibility of directly applying Theorem \ref{th-shen} to $g_n,g_{0,n}\in\G_{a_n}$, even though they have a positive lower bound $a_n$, which arises from two violations of its conditions:
\begin{enumerate}
    \item[1.] For condition \eqref{th-shen-3}, the neighborhood 
    \[
        S_{a_n}(t)
  :=\left\{g\in\G_{a_n}\left|\, \max\{D(g_{0,n},g),V(g_{0,n},g)\}\leq t\right.\right\}
  \]
 shrinks as $a_n\to0$. The prior concentration must therefore be adapted according to $a_n$.
    \item[2.] For condition \eqref{th-shen-1}, this violation directly  impairs the convergence rate. If relating the Hellinger distance with the $L_{2,w}$-norm similarly as in \eqref{l2-h}, then $     d_H^2(u_i,l_i)\leq\frac1{a_n}\Vert u_i-l_i \Vert_{L_{2,w}}^2$. By \eqref{proof-th1-H}, we obtain
\begin{align*}
    \int_{r_n^2/2^8}^{\sqrt{2}r_n}
H_{[~]}^{1/2}\left(\varepsilon,\Omega^{\G_{a_n}}_n,d_H\right)d\varepsilon
  &\lesssim
\int_{r_n^2/2^8}^{\sqrt{2}r_n}H_{[~]}^{1/2}\left(\sqrt{a_n}\varepsilon,\Omega^{\G_{a_n}}_n,L_{2,w}\right)d\varepsilon\\
&\lesssim a_n^{-\frac{1}{4p}}\int_{r_n^2/2^8}^{\sqrt{2}r_n}
  \varepsilon^{-\frac{1}{2p}}d\varepsilon
  \lesssim a_n^{-\frac{1}{4p}}r_n^{1-\frac1{2p}}.
\end{align*}
The condition \eqref{th-shen-1} holds, only if $r_n\gtrsim a_n^{-\frac1{2(2p+1)}}n^{-\frac p{2p+1}}$. Consequently, the convergence rate $\varepsilon_n=\max\{r_n,t_n^{1/2}\}\geq r_n$ in Theorem \ref{th-shen} is strictly slower than the optimal one $\varepsilon_n=n^{-\frac p{2p+1}}$, since $a_n \to 0$ as $n \to \infty$.  
\end{enumerate}

To address these issues, we state a modified version of Theorem \ref{th-shen} adapted to $g_n,g_{0,n}\in\G_{a_n}$.

\begin{theorem}[Modification Theorem \ref{th-shen}]\label{th-modified shen}
Suppose that there exists a sequence $r_n\to0$ and a decreasing sequence $a_n\ll r_n^2$, and a sequence of subsets $\Omega^{\G_{a_n}}_n\subset\Omega^{\G_{a_{n+1}}}_{n+1}\subset\cdots\subset\Omega^{\G}$, such that
\begin{equation}\label{eqn-3.20}
  \int_{r_n/2^8}^{\sqrt{2}r_n}
  H^{1/2}_{[~]}\!\left(u,\Omega^{\G_{a_n}}_n,L_{2,w}\right)\,du
  \lesssim \sqrt{n}r_n^2,
\end{equation}
and there exist sequences $b_n\to\infty$ and $t_n\to0$ such that $a_nb_n\leq n$, $a_nb_nt_n\to\infty$, and
\begin{equation}\label{eqn-3.21}
  \pi\left(S_{a_n}(t_n)\right)\gtrsim \exp(-b_na_nt_n).
\end{equation}
Let $\varepsilon_n=\max\{r_n,2\sqrt{2t_n}\}$, and assume further that
\begin{equation}\label{eqn-3.22}
  \pi\left(\left(\Omega^{\G_{a_n}}_n\right)^c\right)\leq \exp\left(-n\varepsilon_n^2\right).
\end{equation}
Then for sufficiently large $K>0$, if $n\varepsilon_n^2\to\infty$, it holds that
\begin{equation}\label{eqn-3.29}
  \pi\left(A_{n,\btheta_n}^c|\bY^n\right)\lesssim \exp(-Cn\varepsilon_n^2)\to0,
\end{equation}
where $A_{n,\btheta_n}^c$ is defined in \eqref{def-An-theta-n}.
\end{theorem}

Notice that the prior concentration condition \eqref{eqn-3.21} is adapted to $a_n$, and condition \eqref{th-shen-1} of Theorem \ref{th-shen} is replaced by the entropy with respect to the $L_{2,w}$-distance (instead of the Hellinger distance), which completely eliminates the deterioration of the convergence rate. The proof of Theorem \ref{th-modified shen} is analogous to Theorem 4 in \cite{shen2001} (Theorem \ref{th-shen}); its key step is a uniform exponential bound for likelihood ratios. We defer the proof of Theorem \ref{th-modified shen} until after establishing this uniform bound in Lemma \ref{thm-95} below.

\begin{lemma}[exponential bound for likelihood ratios]\label{thm-95}
Given a positive decreasing sequence $a_n\downarrow0$, then $g_{0,n}\in \G_{a_n}$ is shifted from $g_0$ by \eqref{tilde}. For any $\epsilon_n>0$, assume that
\begin{equation}\label{eqn-3.171}
  \int_{\epsilon_n/2^8}^{\sqrt{2}\epsilon_n}
  H^{1/2}_{[~]}\left(u,\Omega^{\G_{a_n}},L_{2,w}\right)\,du
  \lesssim \sqrt{n}\,\epsilon_n^2.
\end{equation}
Then there exist constants $C_4,C_5>0$, independent of $n$, such that
\begin{align}\label{eqn-3.20-lemma}
  &\P_0^*\!\left(
  \sup_{\{g\in\G_{a_n}|\, d_H(g,g_0)\geq \epsilon_n\}}
  \prod_{i=1}^n\frac{g(Y_i)}{g_{0,n}(Y_i)}
  \geq \exp\!\left(-\frac{1}{2}C_4n\epsilon_n^2\right)
  \right) \notag\\
  &\hspace{2cm}\lesssim
  \exp\!\left(-C_5n\epsilon_n^2\right)
  + (1-a_n\gamma_0)^{-n}\exp\!\left(-\frac{1}{2}C_4n\epsilon_n^2\right).
\end{align}
\end{lemma}

\begin{proof}
This argument is similar as those in Theorem 3 (i), \cite{wong-shen-1995}. Let $g\in\G_{a_n}$, then
     \begin{align}\label{eqn-3.6}
   \P_0^*\left(\sup_{\{g\in\G_{a_n}|\,d_H(g,g_0)\geq \epsilon_n\}}\prod_{i=1}^n\frac{g(Y_i)}{g_{0,n}(Y_i)}\geq\exp\left(-\frac12 C_4n\epsilon_n^2\right)   \right)\leq P_1+P_2,
    \end{align}
    where 
    \begin{align*}
P_1:=&\P_0^*\left(\sup_{\{g\in\G_{a_n}|\,d_H(g,g_0)\geq\epsilon_n\}}\prod_{i=1}^n\frac{g(Y_i)}{g_0(Y_i)}\geq\exp(-C_4n\epsilon_n^2)\right),
    \end{align*}
    and 
    \begin{align*}
P_2:=\P_0\left(\prod_{i=1}^n\frac{g_0(Y_i)}{g_{0,n}(Y_i)}\geq\exp\left(\frac12C_4n\epsilon_n^2\right)\right).
    \end{align*}

 To bound $P_2$, using the Markov inequality, one has
    \begin{align}\label{eqn-3.19}\notag
    P_2=&\P_0\left(\prod_{i=1}^n\frac{g_0(Y_i)}{g_{0,n}(Y_i)}\geq\exp\left(\frac12C_4n\epsilon_n^2\right)\right) 
    \leq\exp\left(-\frac12C_4n\epsilon_n^2\right)\prod_{i=1}^n\E_0\left(\frac{g_0(Y_i)}{g_{0,n}(Y_i)}\right)\\
    \leq&(1-a_n\gamma_0)^{-n}\exp\left(-\frac12C_4n\epsilon_n^2\right),
    \end{align}
  where the last inequality follows from 
    \begin{align*}
    \E_0\left(\frac{g_0(Y_i)}{g_{0,n}(Y_i)}\right)
    =&\int_I\frac{g_0^2}{g_{0,n}}wdx=\int_I\frac{g_0^2}{(1-a_n\gamma_0)g_0+a_n}wdx
    \leq\frac1{1-a_n\gamma_0}\int_Ig_0wdx\\
    =&\frac1{1-a_n\gamma_0}.
    \end{align*}
    
To bound $P_1$, we shall use Theorem 1, \cite{wong-shen-1995}, which asserts that for any $\epsilon>0$, if 
\begin{align}\label{mod-th1-int}
    \int_{\epsilon^2/2^8}^{\sqrt2\epsilon}
H^{1/2}_{[~]}\left(u,\Omega^{\G_{a_n}},d_H\right)du
\lesssim  \sqrt{n}\epsilon^2,
\end{align}
then $\displaystyle \P^*\left(
\sup_{\substack{\{g\in\mathcal \G_{a_n}|\,d_H(g,g_0)\ge \epsilon\}}}
\prod_{i=1}^n\frac{g(Y_i)}{g_{0}(Y_i)}
\geq \exp(-C_4 n\epsilon^2)
\right)
\lesssim\exp(-C_5 n\epsilon^2)$.
If we can verify \eqref{eqn-3.171} implies \eqref{mod-th1-int}, then by letting $\epsilon=\epsilon_n$, \begin{equation}\label{eqn-3.18}
        P_1\lesssim\exp\{-C_5n\epsilon_n^2\}.
    \end{equation}
    
To show \eqref{eqn-3.171}$\Rightarrow$\eqref{mod-th1-int}, let us start from the relation between the Hellinger distance and the $L_{2,w}$-norm: for any $f_1,f_2\in\G$, 
\begin{align}\label{eqn-3.3}
    d_H^2(f_1,f_2) \leq &\Vert f_1-f_2\Vert_{L_{1,w}}=\int _I \left|f_1(x)-f_2(x)\right|w(x)dx\\\notag
    \leq& \left(\int _I\left(f_1(x)-f_2(x)\right)^2w(x)dx\right)^{1/2}\,dx
    \left(\int _Iw(x)dx\right)^{1/2}
    =\sqrt{\gamma_0} \Vert f_1-f_2\Vert_{L_{2,w}}.
\end{align}
Consequently, 
\begin{align*}
    \int_{\epsilon^2/2^8}^{\sqrt2\epsilon}H_{[~]}^{1/2}(u,\Omega^{\G_{a_n}},d_H)du
    \leq&\int_{\epsilon^2/2^8}^{\sqrt2\epsilon}H_{[~]}^{1/2}(
    \epsilon^2/2^8,\Omega^{\G_{a_n}},d_H)du\\
   \overset{\eqref{eqn-3.3}} \lesssim&\int_{\epsilon^2/2^8}^{\sqrt2\epsilon}H_{[~]}^{1/2}(
    \sqrt2\epsilon,\Omega^{\G_{a_n}},L_{2,w})du
    \leq\int_{\epsilon^2/2^8}^{\sqrt2\epsilon}\sqrt n\epsilon du
    \leq\sqrt n\epsilon^2,
\end{align*}
where the last inequality but one is due to 
\begin{align*}
     H^{1/2}_{[~]}(\sqrt2\epsilon,\Omega^{\G_{a_n}},L_{2,w})
    \lesssim&
    \frac{1}{\epsilon}
    \int_{\epsilon/2^8}^{\sqrt2\epsilon} H^{1/2}_{[~]}(\sqrt2\epsilon,\Omega^{\G_{a_n}},L_{2,w})du\\
    \lesssim
    & \frac{1}{\epsilon}
    \int_{\epsilon/2^8}^{\sqrt2\epsilon} H^{1/2}_{[~]}(u,\Omega^{\G_{a_n}},L_{2,w})du\overset{\eqref{eqn-3.171}}\leq\sqrt n\epsilon,
\end{align*}
by the decreasing monotonicity of $H_{[~]}(u,\Omega^{\G_{a_n}},L_{2,w})$ in $u$. Finally, \eqref{eqn-3.20-lemma} follows immediately from \eqref{eqn-3.6}, \eqref{eqn-3.19} and \eqref{eqn-3.18}.
\end{proof}

We now proceed to show Theorem \ref{th-modified shen}. The argument is similar as that in Theorem \ref{th-shen}, or Theorem 2, \cite{shen2001}.
\begin{proof}[Proof of Theorem \ref{th-modified shen}]
  Let us write
\begin{align}\label{eqn-3.28}
    \pi\left(A_{n,\btheta_n}^c|\bY^n\right)\overset{\eqref{prosterior}}=\frac{m_n\left(A_{n,\btheta_n}^c,\bY^n\right)}{m_n\left(\Omega,\bY^n\right)}
    =\frac{m_n\left(A_{n,\btheta_n}^c\cap\Omega_n,\bY^n\right)}{m_n\left(\Omega,\bY^n\right)}+\frac{m_n\left(A_{n,\btheta_n}^c\cap\Omega_n^c
    ,\bY^n\right)}{m_n\left(\Omega,\bY^n\right)},
\end{align}
where $\Omega_n$ is in \eqref{eqn-O_M} and
$\displaystyle m_n\left(A,\bY^n\right):=\int_A\prod_{i=1}^n\frac{g(Y_i)}{g_{0,n}(Y_i)}d\pi(\bbeta)$. Similar as the arguments in Theorem 2, \cite{shen2001}, we shall bound the numerators \linebreak[4]$m_n\left(A_{n,\btheta_n}^c\cap\Omega_n,\bY^n\right)$, $m_n\left(A_{n,\btheta_n}^c\cap\Omega_n^c,\bY^n\right)$ and the denominator $m_n\left(\Omega,\bY^n\right)$ in \eqref{eqn-3.28}, separately.

Let us bound the numerator $m_n\left(A_{n,\btheta_n}^c\cap\Omega_n,\bY^n\right)$ first. By the triangular inequality, we have
\[
    Kr_n\leq d_H(g_{0,n},g)\leq d_H(g_{0,n},g_0)+d_H(g_0,g)\leq \sqrt{2\gamma_0a_n}+d_H(g_0,g),
\]
since
\begin{align*}
    d^2_H(g_{0,n},g_0)=&2-2\int_I\sqrt{g_0g_{0,n}}wdx
    =2-2\int_I\sqrt{g_0((1-a_n\gamma_0)g_0+a_n)}wdx\\
    \leq&2-2\int_I\sqrt{1-a_n\gamma_0}g_0wdx
    =2-2\sqrt{1-a_n\gamma_0}\leq2\gamma_0a_n.
\end{align*}
Thus, for any $g\in A_{n,\btheta_n}^c\cap\Omega_n$, 
\begin{equation}\label{eqn-3.30}
    d_H(g_0,g)\geq Kr_n-\sqrt{2\gamma_0a_n}.
\end{equation}

From \eqref{eqn-3.20-lemma} with $\varepsilon=Kr_n-\sqrt{2\gamma_0a_n}$, by Lemma \ref{thm-95}, one has 
\begin{align}\label{eqn-3.24}\notag
    &\P_0^*\left(\sup_{\left\{g\in\Omega^{\G_{a_n}}_n|\,d_H(g,g_0)\geq Kr_n-\sqrt{2\gamma_0a_n}\right\}}\prod_{i=1}^n\frac{g(Y_i)}{g_{0,n}(Y_i)}\geq\exp\left(-\frac12 C_1n(Kr_n-\sqrt{2\gamma_0a_n})^2\right)   \right)\\
    \lesssim&\exp(-Cn(Kr_n-\sqrt{2\gamma_0a_n}))^2). 
\end{align}
Therefore, 
\begin{align}\label{eqn-3.25}\notag
   & m_n\left(A_{n,\btheta_n}^c\cap\Omega_n,\bY^n\right)
    =\int_{A_{n,\btheta_n}^c\cap\Omega_n}\prod_{i=1}^n\frac{g(Y_i)}{g_{0,n}(Y_i)}d\pi(\bbeta)\\\notag
   & \overset{\eqref{eqn-3.30},\eqref{eqn-3.24}}\leq\int_{\left\{g\in\G_{a_n}|\,d_H(g,g_0)\geq Kr_n-\sqrt{2\gamma_0a_n}\right\}} \sup_{\left\{g\in\G_{a_n}|\,d_H(g,g_0)\geq Kr_n-\sqrt{2\gamma_0a_n}\right\}} \prod_{i=1}^n\frac{g(Y_i)}{g_{0,n}(Y_i)} d\pi(\bbeta)\\
   & \leq\exp\left(-\frac12 C_1n(Kr_n-\sqrt{2\gamma_0a_n}))^2\right),
\end{align}
with probability tending to one, if $n(r_n-\sqrt{a_n})^2\rightarrow\infty$.

The lower bound of $m_n\left(\Omega,\bY^n\right)$ is exactly the same as those in Lemma 1, Section 3, \cite{shen2001}. We state the result here. 
\begin{align*}
    \P_0\left(m_n\left(\Omega,\bY^n\right)\leq\frac12\pi\left(S_{a_n}(t_n)\right)e^{-2nt_n}\right)\lesssim\frac1{nt_n}.
\end{align*}
Consequently, there exists positive sequence $b_n\to\infty$, such that $(a_nb_n+n)t_n\to\infty$, then
\begin{equation}\label{eqn-3.26}
    m_n(\Omega,\bY^n)\geq\frac12\pi(S_{a_n}(t_n))e^{-2nt_n}\overset{\eqref{eqn-3.21}}\gtrsim\exp(-2nt_n-b_na_nt_n),
\end{equation}
except on a set of probability tending to $0$.

Lastly, the upper bound of $m_n(A_{n,\btheta_n}^c\cap\Omega_n^c,\bY^n)$ follows the same argument in Theorem 4, \cite{shen2001}. By the Markov's inequality and Fubini's theorem,  we have
\begin{align}\label{79}\notag
&\P_0\left(m_n\left(A_{n,\btheta_n}^c\cap\Omega_n^c,\bY^n\right)\geq e^{-\frac12n\varepsilon_n^2}\right)
    \leq e^{\frac{n\varepsilon_n^2}2}\E_0\left[m_n\left(A_{n,\btheta_n}^c\cap\Omega_n^c,\bY^n\right) \right]\\
    \notag
=&e^{\frac{n\varepsilon_n^2}2}\int_{A_{n,\btheta_n}^c\cap\Omega_n^c}\int_{\mathcal{Y}^n}\prod_{i=1}^n\frac{g(y_i)}{g_{0,n}(y_i)}d\P_0(y^n)d\pi(\bbeta)\\\notag
=&e^{\frac{n\varepsilon_n^2}2}\int_{A_{n,\btheta_n}^c\cap\Omega_n^c}\int_{\mathcal{Y}^n}\prod_{i=1}^n\frac{g_0(y_i)}{g_{0,n}(y_i)}g(y_i)w^n(y_1)dy^nd\pi(\bbeta)\\
  \overset{\eqref{tilde}} \leq&e^{\frac{n\varepsilon_n^2}2}(1-a_n\gamma_0)^{-n}\int_{A_{n,\btheta_n}^c\cap\Omega_n^c}\int_{\mathcal{Y}^n}\prod_{i=1}^ng(y_i)w^n(y_1)dy^nd\pi(\bbeta)\\\notag
=&e^{\frac{n\varepsilon_n^2}2}(1-a_n\gamma_0)^{-n}\pi(A_{n,\btheta_n}^c\cap\Omega_n^c)
    \leq 
    e^{\frac{n\varepsilon_n^2}2}
    (1-a_n\gamma_0)^{-n}
    \pi\left(\Omega ^{\G_{a_n}}_n\right)^c
    \overset{\eqref{eqn-3.22}}
    \lesssim 
    e^{-\frac12n\varepsilon_n^2+na_n\gamma_0}.
\end{align}
Thus, 
\begin{equation}\label{eqn-3.27}
m_n\left(A_{n,\btheta_n}^c\cap\Omega_n^c,\bY^n\right)\leq e^{-\frac12n\varepsilon_n^2}
\end{equation}
in probability, if $a_n\ll\varepsilon_n^2$. 

Therefore, by letting $\varepsilon_n=\max\{r_n,2\sqrt{2t_n}\}$, \eqref{eqn-3.29} follows immediately by substituting \eqref{eqn-3.25}--\eqref{eqn-3.27} back into \eqref{eqn-3.28}, i.e.
\begin{align*}
    \pi\left(A_{n,\btheta_n}^c|\bY^n\right)
    \lesssim&\exp\left(-\frac12Cn(K\varepsilon_n-\sqrt{a_n})^2+(2n+a_nb_n)\frac{\varepsilon_n^2}8\right)\\
    & +\exp\left(-\frac12n\varepsilon_n^2+(2n+a_nb_n)\frac{\varepsilon_n^2}8\right)
    \lesssim\exp(-Cn\varepsilon_n^2),
\end{align*}
if $a_n\ll\varepsilon_n^2$ and $a_nb_n\ll n$.
  \end{proof}

Finally, we arrive at the analogous result in Theorem \ref{th-bound} for $g_0\in \G$ without the strictly positive lower bound.
  
\begin{theorem}\label{th-prior-general}
Let $g_0$ and $g\in\G$. The corresponding coefficients are $\btheta=(\theta_0,\theta_1,\cdots),\,\bbeta=(\eta_0,\eta_1,\cdots)\in\Omega^\G$, associated with $\G$. For any $\varepsilon_n\geq n^{-p/(2p+1)}$ and any integer $p\geq1$, suppose there exists a positive decreasing sequence $a_n\downarrow0$ and $a_n\lesssim \varepsilon_n^4$. 
Assume that there exists a sequence $b_n\to\infty$ satisfying $a_nb_n\leq n$, $a_nb_n\varepsilon_n^2\to\infty$ and a constant $C_6>0$, independent of $n$, such that
\begin{equation}\label{eqn-4.1}
  \pi\left(
  \left\{\bbeta\in\Omega_n^\G\left|\, \sum_{j=0}^{\infty}(\eta_j-\theta_j)^2\gamma_j \leq \frac{C_6a_n\varepsilon_n^2}{16(1-a_n\gamma_0)^2}\right.\right\}
  \right)
  \gtrsim \exp\left(-\frac{1}{8}C_6 a_n b_n\varepsilon_n^2\right),
\end{equation}
and
\begin{equation}\label{eqn-4.2}
  \pi\left(
  \left\{\bbeta\in\Omega^\G\left|\, \sum_{j=0}^{\infty}\eta_j^2\tilde\gamma_j>\frac {n^{1/(2p+1)}}{(1-a_n\gamma_0)^2}\right.\right\}
  \right)
  \lesssim \exp\left(-C_6 n\varepsilon_n^2\right),
\end{equation}
where $\pi$ be a prior on $\bbeta=(\eta_0,\eta_1,\cdots)\in\Omega^\G$, and $\tilde\gamma_j$ is defined in \eqref{eqn-2.3}. If $n\varepsilon_n^2\to\infty$, then for sufficiently large $K>0$,
\begin{equation}\label{eqn-4.3}
  \pi(A_{n}^c|\bY^n)
  \lesssim
  \exp\left(-\frac{nK^2\varepsilon_n^2}{2}\right)
  +\exp\left(-\frac{C_6n \varepsilon_n^2}{4}\right)
  \to0,
\end{equation}
except on a set of probability tending to $0$, where
$A_n^c$ is given in \eqref{An-c}.
\end{theorem}
\begin{remark}
   \begin{enumerate}
    \item[(1)] Theorem \ref{th-prior-general}  is more widely applicable than Theorem \ref{th-bound}, since condition \eqref{eqn-3.1} in Theorem \ref{th-bound} fails if the true density $g_0(x)$ has no strictly positive lower bound, i.e. $a=0$.
    \item[(2)] Condition \eqref{eqn-4.1} in Theorem \ref{th-prior-general} is stronger than \eqref{eqn-3.1} in Theorem \ref{th-bound}, since
    \begin{align*}
        \pi\left\{\bbeta\left|\,\sum_{j=0}^\infty(\eta_j-\theta_j)^2\gamma_j\leq C_3a\varepsilon_n^2\right.\right\}
        \geq& \pi\left\{\bbeta\left|\,\sum_{j=0}^\infty(\eta_j-\theta_j)^2\gamma_j\leq\frac{C_6a_n\varepsilon_n^2}{16(1-a_n\gamma_0)^2}\right.\right\}\\
        \overset{\eqref{eqn-4.1}}\gtrsim&\exp(-C_6a_nb_n\varepsilon_n^2)\gtrsim\exp(-C_3n\varepsilon_n^2).
    \end{align*}
This is the cost of whether the absence of a uniform positive lower bound is recognized.
    \end{enumerate}
\end{remark}
\begin{proof}[The sketch proof of Theorem \ref{th-prior-general}]
The proof is similar as that of Theorem~\ref{th-bound}, except that Theorem \ref{th-modified shen} (instead of Theorem \ref{th-shen}) is applied to the shifted pair $(g_n,g_{0,n})$ to obtain the convergence rate $\varepsilon_n$ of the shifted posterior $\pi\left(A_{n,\btheta_n}^c|\bY^n\right)$. Then Lemma \ref{lem-3.1} guarantee the same convergence rate of the original posterior $\pi\left(A_n^c|\bY^n\right)$, if $a_n\lesssim\varepsilon_n^4$.
\end{proof}
Theorem~\ref{th-prior-general} shows that posterior convergence at rate $\varepsilon_n$ under Hellinger loss follows from two conditions on the prior for the expansion coefficients: (i) sufficient prior mass in a weighted $L_2$-neighborhood of the true coefficients (weighted by $\gamma_j$), and (ii) exponentially small prior mass outside an appropriate sieve set (measured by $\tilde\gamma_j$). This highlights why, in many existing works, the prior on coefficients is constructed after fixing a specific basis.

\section{Construction of Sieve Priors}\label{sec-4}

In this section, we shall construct a sieve prior on $\bbeta\in\Omega^\G$ which satisfies the conditions in Theorem \ref{th-prior-general}, yielding the minimax optimal convergence rate $\varepsilon_n=n^{-p/(2p+1)}$.

\begin{theorem}\label{th-sieve-gaussian}
Let the parameter space $\Omega$ defined in \eqref{def-Omega} with further assumption that $\tilde\gamma_j\geq j^{7p}\gamma_j$, $j\geq1$, and $\pi$ be the sieve prior of $\bbeta\in\Omega$ of the form
\begin{equation}\label{pi}
  \pi(A):=\sum_{i=0}^{k_n} p_i\pi_{<i}(A),
\end{equation}
for any measurable set $A$, where $p_i\geq0$ are weights with $\sum_{i=0}^{k_n} p_i=1$, and $\pi_{<i}(\bbeta_{<i})=\prod_{j=1}^{i-1}\pi(\eta_j)$ is the truncated Gaussian sieve of $\bbeta_{<i}:=(\eta_0,\cdots,\eta_{i-1})$, i.e.
\begin{align}\label{gau-beta}
  \eta_0 &= 1, \notag\\
  \eta_j &\sim \N\!\left(0,\,j^{-2p}\gamma_j^{-1}\right),
  \qquad \textup{for } 1\leq j\leq i-1, \notag\\
  \eta_j &= 0,\qquad j\geq i,
\end{align}
for $i\leq k_n=\O\left(n^{\frac{6p+1}{7p(2p+1)}}\right)$, with $\gamma_j$ being the normalization constant in \eqref{p}. Then \eqref{eqn-4.3} in Theorem \ref{th-prior-general} holds for sufficiently large $K>0$.
\end{theorem}
\begin{proof}
We verify the two key ingredients: (i) a prior-mass lower bound near the true density \eqref{eqn-4.1} and (ii)  a tail bound \eqref{eqn-4.2}. For convenience, we work with the shifted coefficient $\bbeta_n$. By \eqref{eqn-2.1}, it is easy to transform the prior of $\bbeta$ to that of $\bbeta_n$:
\begin{align*}
    \eta_{0,n}=&a_n+(1-a_n\gamma_0),\\
    \eta_{j,n}\sim& \N\left(0,(1-a_n\gamma_0)^2j^{-2p}\gamma_j^{-1}\right),\qquad \textup{for } 1\leq j\leq i-1,\\
    \eta_{j,n}=&0,\qquad j\geq i,
\end{align*} 
for $i\leq k_n$, so that $\pi_{<i}(\bbeta_{<i,n})=\prod_{j=1}^{i-1}\pi(\eta_{j,n})$ is the truncated Gaussian sieve of $\bbeta_{<i,n}$.

\begin{enumerate}
\item[(1)] \textit{For \eqref{eqn-4.1}:} 
Written in $\bbeta_n$, \eqref{eqn-4.1} is equivalent to
\begin{equation*}
  \pi\left(
  \left\{\bbeta_n\in\Omega_n^{\G_{a_n}}\left|\, \sum_{j=0}^{\infty}(\eta_{j,n}-\theta_{j,n})^2\gamma_j \leq \frac{C_6a_n\varepsilon_n^2}{16}\right.\right\}
  \right)
  \gtrsim \exp\left(-\frac{1}{8}C_6 a_n b_n\varepsilon_n^2\right).
\end{equation*}
As before in the proof of Theorem \ref{th-bound}, let us define the neighborhood 
\begin{equation}\label{eqn-4.15}
    \bar{S}_{a_n}(\varepsilon_n^2):=\left\{\bbeta_n\in\Omega_n^{\G_{a_n}}\left|\, \sum_{j=0}^\infty(\eta_{j,n}-\theta_{j,n})^2\gamma_j \leq a_n\varepsilon_n^2\right.\right\},
\end{equation}
analogously as in \eqref{Soverline}. By the definition of the sieve prior \eqref{pi}, one has
\begin{equation}\label{eqn-4.11}
  \pi\left(\bar{S}_{a_n}(\varepsilon_n^2)\right)
  =\sum_{j=0}^{k_n}p_j\pi_{<j}\!\left(\bar{S}_{a_n}(\varepsilon_n^2)\right)
  \geq p_{k_n}\pi_{<k_n}\left(\bar{S}_{a_n}(\varepsilon_n^2)\right).
\end{equation}
It is clear that $\bar S_{a_n}(t)\subset S_{a_n}(t)$, for any $t>0$, as argued in step (2) of the proof of Theorem \ref{th-bound}. Thus, we only need to show 
\begin{equation}\label{eqn-4.16}
    \pi_{<k_n}(\bar S_{a_n}(\varepsilon_n^2))\gtrsim\exp\left(-\frac18C_6a_nb_n\varepsilon_n^2\right).
\end{equation}

To show \eqref{eqn-4.16}, we split the series in $\bar S_{a_n}(\varepsilon_n^2)$ into two parts: the one sums from $k_n$ to $\infty$, and the other one sums from $1$ to $k_n-1$. 
Notice that
\begin{align}\label{eqn-4.13}\notag
    \sum_{j=k_n}^{\infty}(\theta_{j,n}-\eta_{j,n})^2\gamma_j
    \leq& \sum_{j=k_n}^{\infty}(\theta_{j,n}-\eta_{j,n})^2\frac{\tilde\gamma_j}{j^{7p}}
  \leq \frac{1}{k_n^{7p}}\sum_{j=k_n}^{\infty}(\theta_{j,n}-\eta_{j,n})^2\tilde\gamma_j\\
  \leq &\frac{1}{k_n^{7p}}n^{\frac{1}{2p+1}}
  \leq \tfrac{1}{2}a_n\varepsilon_n^2,
\end{align}
since $\bbeta_n,\btheta_n\in\Omega_n$ \eqref{eqn-O_M}, provided that $k_n=\O\left(n^{\frac{6p+1}{7p(2p+1)}}\right)$, $a_n=\varepsilon_n^4$ and $\tilde\gamma_j\geq j^{7p}\gamma_j$ for $j\geq1$. Thus,
\begin{equation}\label{eqn-4.14}
  \pi_{<k_n}\left(\bar{S}_{a_n}(\varepsilon_n^2)\right)
  \gtrsim \pi_{<k_n}(B_{k_n}),
\end{equation}
where $\displaystyle B_{k_n}
  :=\left\{\bbeta_n\in\G^{a_n}_n\left|\, \sum_{j=0}^{k_n-1}(\theta_{j,n}-\eta_{j,n})^2\gamma_j \leq \tfrac{1}{2}a_n\varepsilon_n^2\right.\right\}$.

The remainder of (i) is to derive the lower bound for $\pi_{<k_n}(B_{k_n})$. The derivation is straightforward but lengthy. To avoid distraction, we only sketch the main steps and relegate the details to Section C of the supplementary material \cite{LL:26}.
\begin{align}\label{eqn-4.17}\notag
        \pi_{<k_n}(B_{k_n})
=& \int_{B_{k_n}}
\prod_{i=0}^{k_n-1}
\frac{i^p\,\gamma_i^{1/2}}{\sqrt{2\pi}\,(1-a_n\gamma_0)}
\exp\left\{
-\frac{\eta_{i,n}^2}
{2 i^{-2p}
\gamma_i^{-1} (1-a_n\gamma_0)^2}
\right\}
d\bbeta_{<k_n,n}\\\notag
\gtrsim&
\exp{\left\{ -\frac{n^{{1}/{(2p+1)}}}{(1-a_n\gamma_0)^2}
\right\}}
\left(2\pi\right)^{\frac{k_n}{2}}
\left[ (k_n-1)!\right]^p\\\notag
    &\cdot\int_{\left\{  \mathbf u_n \Big| \sum \limits _{j=0}^{k_n-1} u_{j,n}^2 \le  \frac{1}{2}{a_n\varepsilon_n^2}\right\}}
\exp\left\{
-\frac{u_{i,n}^2}
{ i^{-2p}
(1-a_n\gamma_0)^2}
\right\}d\bu_{<k_n,n},\\\notag
    &\phantom{aaaaaaaaaaaaaaaaaaaaaaaaaaaaa}\textup{by\ letting\ }u_{i,n}=\sqrt{\gamma_i}(\eta_{i,n}-\theta_{i,n})\\\notag
    \gtrsim&\exp{\left\{ -\frac{n^{{1}/{(2p+1)}}}{(1-a_n\gamma_0)^2}
\right\}}
\left(2\pi\right)^{\frac{k_n}{2}}
\left[ (k_n-1)!\right]^p\frac{\pi^{k/2}}{\Gamma(k_n/2)}\, 
{\left(\tfrac12 a_n\varepsilon_n^2\right)}^{\frac{k_n}{2}}\\\notag
    &\cdot\int_0^1u^{k_n/2-1}
\exp\left\{
-\frac{\tfrac12 a_n\varepsilon_n^2 k_n^{2p} u}{(1-a_n\gamma_0)^2}
\right\}du\\\notag
    \geq& \exp{\left\{ -\frac{n^{{1}/{(2p+1)}}}{(1-a_n\gamma_0)^2}
\right\}}
\left(2\pi\right)^{\frac{k_n}{2}}
\left[ (k_n-1)!\right]^p\frac{\pi^{k_n/2}}{\Gamma(k_n/2)}
{\left(\tfrac12 a_n\varepsilon_n^2\right)}^{\frac{k_n}{2}}\\\notag
    &\cdot\frac{2}{k_n}\exp\left\{
-\frac{a_n\varepsilon_n^2 k_n^{2p} }{2(1-a_n\gamma_0)^2}
\right\}\\
    \gtrsim&\exp{\left\{-2n^{\frac{1}{2p+1}}\right\}}
    \gtrsim\exp\left(-\frac18C_6 a_n b_n \varepsilon_n^2\right),
\end{align}
if $k_n=\O\left(n^{\frac{6p+1}{7p(2p+1)}}\right)$, $a_n=\varepsilon_n^4$ and $b_n =n^{1+\frac{4p}{2p+1}}$. Therefore, \eqref{eqn-4.16} follows immediately from \eqref{eqn-4.14} and \eqref{eqn-4.17}.

\item[(2)] \textit{For \eqref{eqn-4.2}:} Notice that 
\begin{equation}\label{eqn-4.18}
     \left\{\bbeta\in\Omega^\G\left|\,\sum_{j=0}^\infty\eta_j^2\tilde\gamma_j>\frac{n^{1/(2p+1)}}{(1-a_n\gamma_0^2)}\right.\right\}
    \subset\left\{\bbeta\in\Omega^\G\left|\,\sum_{j=0}^\infty\eta_j^2\tilde\gamma_j>n^{1/(2p+1)}\right.\right\}\overset{\eqref{eqn-O_M}}=:\Omega_n^c,
\end{equation}
and 
\begin{align}\label{pinc}\notag
    \pi\left (  \Omega_n^c\right)
    =&1-\pi(\Omega_n) 
   \leq1-\pi_{<2}(\Omega_n) \\\notag
     \overset{\eqref{gau-beta}}=&1-\left( 2\Phi(v)-1\right),\quad\textup{by\ letting}\ v:=\sqrt{\frac{n^{1/(2p+1)}-\gamma_0}{\gamma_1^{-1}\tilde \gamma_1}}\\
    \leq&\frac{2 \exp{\{-v^2/2\}}}{v}
\lesssim\frac{\exp{\{-C_6    n^{1/(2p+1)} \}}}{n}
\lesssim \exp{ \{-C_6 n \varepsilon_n^2\}},
\end{align}
where $\Phi(\cdot)$ is the distribution function of standard normal distribution. The detailed computation of \eqref{pinc} is in Section C of the supplementary material \cite{LL:26}. Therefore, \eqref{eqn-4.2} follows immediately from \eqref{eqn-4.18} and \eqref{pinc}.
\end{enumerate}
\end{proof}

\section{Numerical experiments}\label{sec-5}

In this section, we present a numerical example to illustrate the performance of the Bayesian density estimator. More numerical experiments on unbounded domains have been included in Section D of the supplementary material \cite{LL:26} due to the page limit.

We consider the log-density studied in \cite{AOS2007} (slightly modified on $[-1,1]$ to better fit the polynomial basis):
\[
     g_0(x)=c_0\exp\left(\sin\left( \frac{x+1}{2}\right)\pi\right)-1,
\]
for $x\in[-1,1]$, where $c_0$ is the normalization constant. The Legendre polynomial $\{L_i(x)\}_{i=0}^\infty$ with the weight $w(x)\equiv1$ is discussed in Section B of the supplementary material \cite{LL:26}. To better cope with the vanishing boundary at $\pm1$, we use the  generalized Legendre polynomials introduced in \cite{Shenjie}:
$\displaystyle \tilde L_j(x)=\frac{L_j(x)-L_{j-2}(x)}{\sqrt{4j+6}}$, which is not orthogonal polynomials. Nevertheless, the expansions of $g(x)$ in terms of $L_i(x)$ or $\tilde L_i(x)$ are equivalent, i.e.
\begin{align}\label{eqn-5.4}
  g(x)
  \approx&\sum_{j=0}^{N}\tilde\eta_j\tilde L_j(x)
  =\frac{\tilde\eta_j}{\sqrt{4j+6}}\left(L_j(x)-L_{j-2}(x)\right)\\ \notag
  =&\sum_{j=0}^{N-2}\left(\frac{\tilde \eta_j}{\sqrt{4j+6}}-\frac{\tilde\eta_{j+2}}{\sqrt{4j+14}}\right)L_j(x)
   +\sum_{j=N-1}^{N}\frac{\tilde\eta_j}{\sqrt{4j+6}}L_j(x)
  :=\sum_{j=0}^{N}\eta_j L_j(x).
\end{align}

Let $\bY^n=(Y_1,\ldots,Y_n)$ be i.i.d. samples drawn from $g_0(x)$. The Bayesian estimator is the posterior mean
\begin{align}\label{bayes}
  \hat g(x|\bY^n)
  &=\int p(x|\bbeta)d \pi(\bbeta|\bY^n)  =\E^{\pi(\bbeta|\bY^n)}\left[g(x|\bbeta)\right]
    \propto \int_{\Omega} g(x|\bbeta)\pi(\bbeta)l_n(\bY^n|\bbeta)d\bbeta,
\end{align}
where $\displaystyle l_n(\bY^n|\bbeta):=\log p(\bY^n|\bbeta)=\prod_{i=1}^{n} g(Y_i|\bbeta)
  =\prod_{i=1}^{n}\left(\sum_{j=0}^{N}\eta_j L_j(Y_i)\right)$
is the likelihood. From Theorem \ref{th-prior-general}, one can easily obtain the Hellinger distance between the Bayesian estimator and the true density:
\begin{corollary}\label{cor1}
Assume that conditions in Theorem \ref{th-prior-general} hold. If $\varepsilon_n=n^{-1/(2p+1)}$, then for sufficiently large $K>0$, 
\begin{align}\label{eqn-cor1}
    d_{H}^2\left(  g_{0},\hat{g}     \right)
    \lesssim 
    \varepsilon_n^2 +2\exp\!\left(-\frac{nK^2\varepsilon_n^2}{2}\right)
  +2\exp\!\left(-\frac{C_6n \varepsilon_n^2}{4}\right),
\end{align}
where $\hat{g}$ is the Bayesian estimator in \eqref{bayes}, $g_0$ is the true density and $C_6$ is the constant in Theorem \ref{th-prior-general}.
\end{corollary}
\begin{proof}
    Following the same argument in the proof of Theorem 5, \cite{shen2001}, one has 
    \begin{align}\label{eqn-5.12}
        d_{H}^2(g_0,\hat g) \le
        \varepsilon_n^2 \pi(A_n|\bY^n)
        +2 \pi(A_n^c|\bY^n)
        \lesssim \varepsilon_n^2
        +2 \pi(A_n^c|\bY^n).
    \end{align}
   The proof is completed by combining \eqref{eqn-5.12} and \eqref{eqn-4.3} in Theorem \ref{th-prior-general}.
\end{proof}

\subsection{Experiment 1: The Bayesian estimator} 

In the sequel, we shall numerically verify that the Bayesian estimator $\hat g$ is a ``good'' approximation to the true density $g_0$ under our setting, i.e. using the Legendre polynomials, or equivalently the generalized Legendre polynomials. Let us set $p=2$, the sample size $n=10000$, then $k_n=9\approx n^{\frac{6p+1}{7p(2p+1)}}-1$, and the Gaussian sieve prior is assigned to $\bbeta=(\eta_0,\eta_1,\cdots)$:
\begin{align}\label{eqn-5.3}
  \eta_j &\sim \N(0,\sigma_j^2), \qquad j=0,1,\cdots,9, \notag\\
  \eta_j &=0, \qquad j\geq 10,
\end{align}
with 
$\{\sigma_j\}_{j=0}^{9}
  =
  [4.03; 5.12; 2.41; 1.68; 1.17; 0.96; 0.64; 0.55; 0.28; 0.25]$.
Or equivalently, the Gaussian sieve prior on $\tilde\bbeta$ is
\begin{align*}
    \tilde\eta_j\sim&\N\left(0,\frac{\sigma_j^2}{4j+6}+\frac{\sigma_{j+2}^2}{4j+14}\right),\quad j=0,1,\cdots,7,\\
    \tilde\eta_j\sim&\N(0,\sigma_j^2),\quad j=8,9,\quad \textup{and}\ 
    \tilde\eta_j=0,\quad j\geq10.
\end{align*}
The posterior mean in \eqref{bayes} is obtained by drawing Markov chain Monte Carlo (MCMC) samples from the posterior distribution induced by the prior \eqref{eqn-5.3}. A random walk Metropolis algorithm is employed, generating $10000$ iterations where the first $2000$ are treated as burn-in.

\begin{figure}[t]
    \centering
    \begin{subfigure}{0.6\linewidth}
              \hspace{-5pt}\includegraphics[scale=0.3]{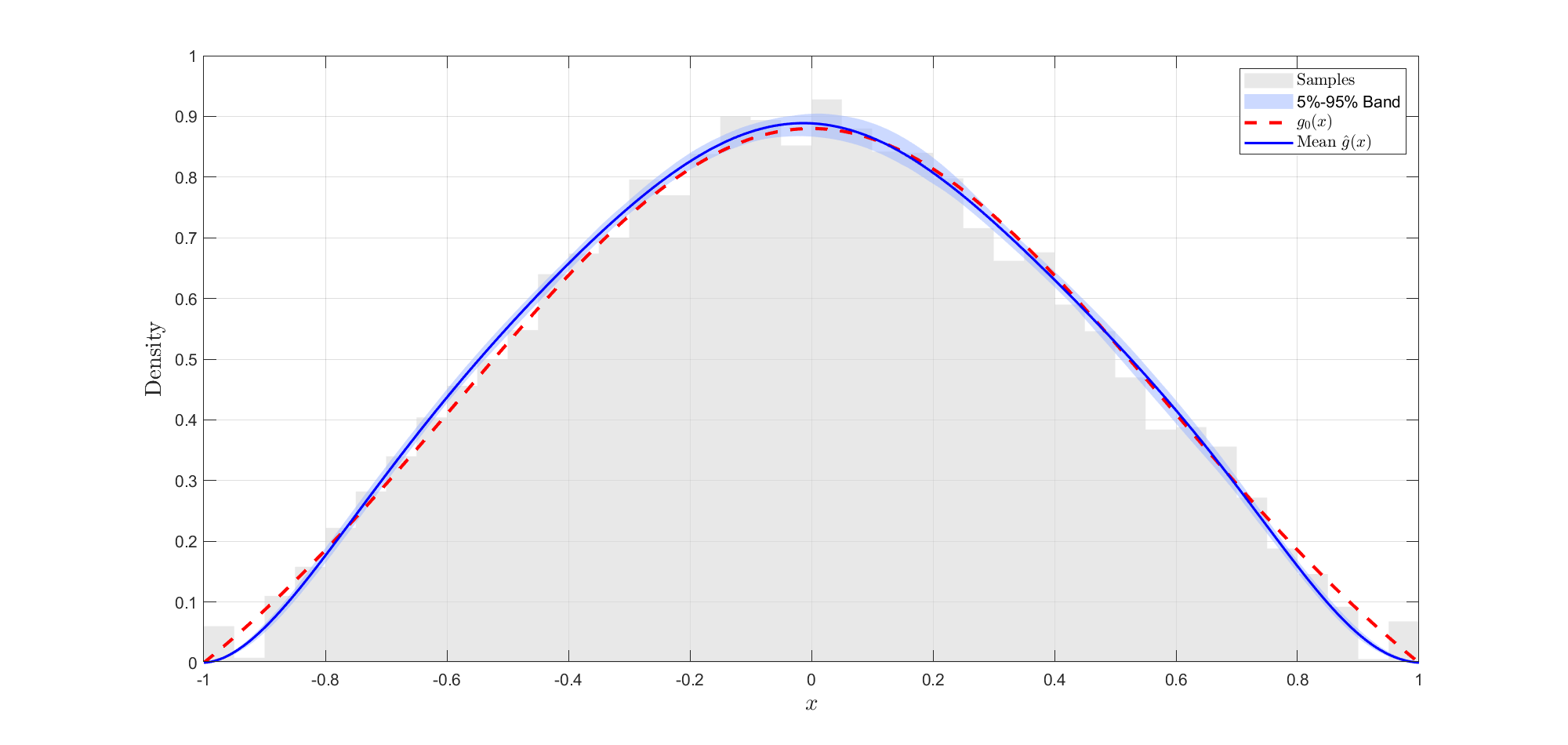}
        \caption{Using the generalized Legendre basis}
        \label{subfig-1.1}
    \end{subfigure}
    \\
    \begin{subfigure}{0.6\linewidth}
          \hspace{-5pt} \includegraphics[scale=0.3]{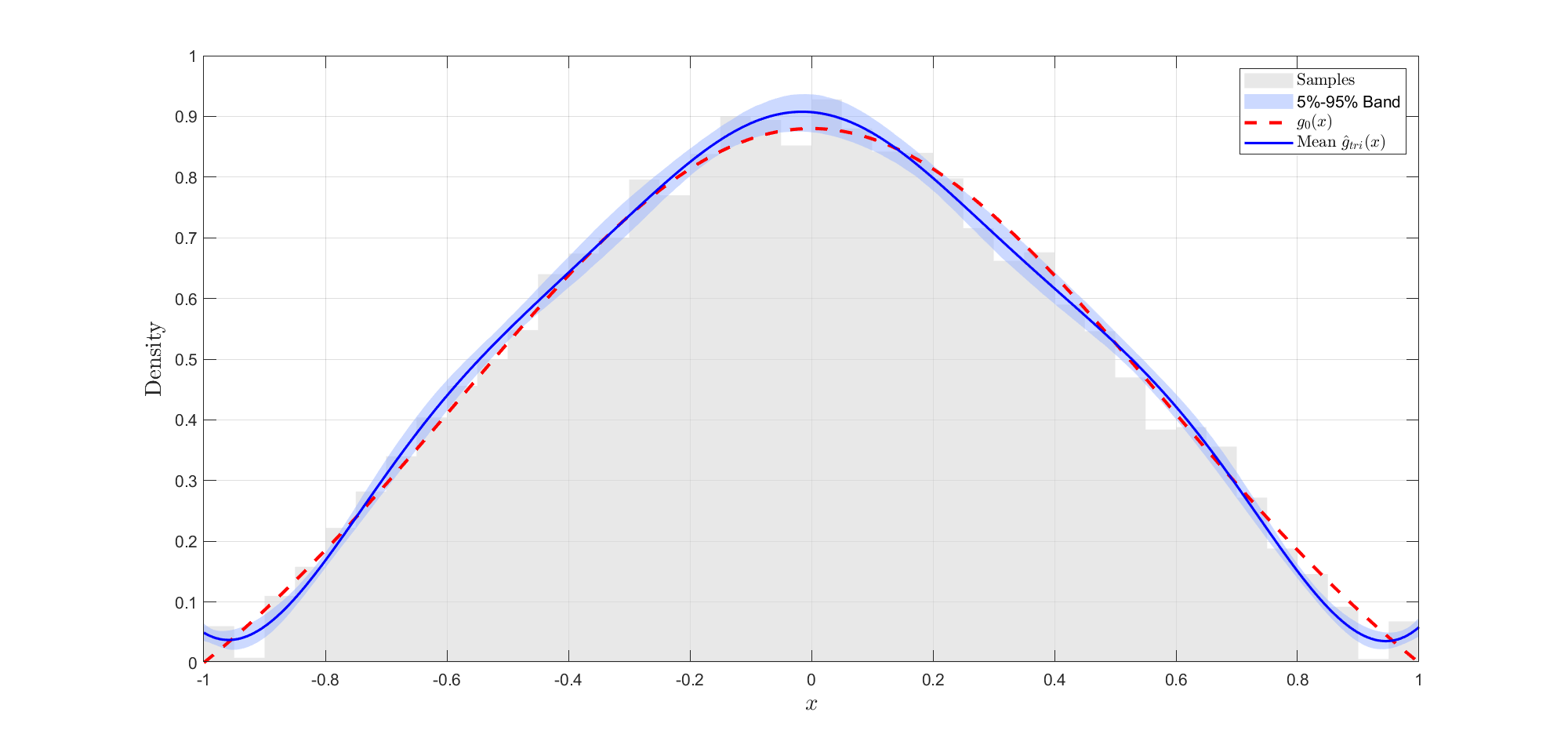}
        \caption{Using the trigonometric basis}
        \label{subfig-1.2}
    \end{subfigure}
    
    \caption{True density $g_0(x)$ (red dashed) and the posterior mean density of $\hat g(\cdot|\bY^n)$ (blue) using the generalized Legendre polynomials and the trigonometric basis, and their $95\%$ posterior confidence bands.}
    \label{fig-1}
\end{figure}

As a comparison, the trigonometric basis in \cite{AOS2007} are also investigated. This estimator $\hat g_{tri}(x)$ has the form
\[
   \hat g_{tri}(x)\approx\eta_{tri,0}+\sqrt{2}\sum_{j=1}^{k}\left\{\eta_{tri,2j-1}\sin(2\pi jx)+\eta_{tri,2j}\cos(2\pi jx)\right\},
\]
where the coefficients are mutually independent with the prior
\begin{equation}\label{eqn-5.13}
    \eta_{tri,0}=0,\quad \eta_{tri,j}\sim \N\left(0,\gamma_j^{-(2p+1)}\right),
\end{equation}
for $j=1,\cdots,k=5$, with
$\gamma_0=0$ and $\gamma_j=
\begin{cases}
j+1, & \text{for } j \text{ odd},\\
j, & \text{for } j \text{ even}.
\end{cases}$
With the same set of samples $\bY^n$ from $g_0$, the Bayesian estimator is obtained by drawing MCMC samples from the posterior distribution induced by the prior \eqref{eqn-5.13}.

The Bayesian estimators $\hat g$ and $\hat g_{tri}$ and the true density $g_0$, as well as their $95\%$ confidence bands have been displayed in Figure \ref{fig-1}. Both estimators, $\hat{g}$ and $\hat{g}_{tri}$, successfully capture the global features of the true density. 
However, the generalized Legendre basis (Figure \ref{subfig-1.1}) demonstrates noticeably superior overall performance compared to the trigonometric basis (Figure \ref{subfig-1.2}). 
The posterior mean $\hat{g}$ aligns almost perfectly with the true density $g_0(x)$ across the entire domain, precisely capturing the mode in the center, i.e. near $x=0$, with small deviation. 
Moreover, its corresponding $95\%$ confidence band is remarkably smooth and uniformly narrow over the whole interval. 
Most notably, the Bayesian estimator using the generalized Legendre polynomials effectively mitigates boundary bias near $x=\pm 1$, whereas that using the trigonometric basis suffers from pronounced boundary artifacts, leading to significantly inflated uncertainty at the boundaries.
\begin{figure}[t]
  \centering
  \includegraphics[width=0.55\linewidth]{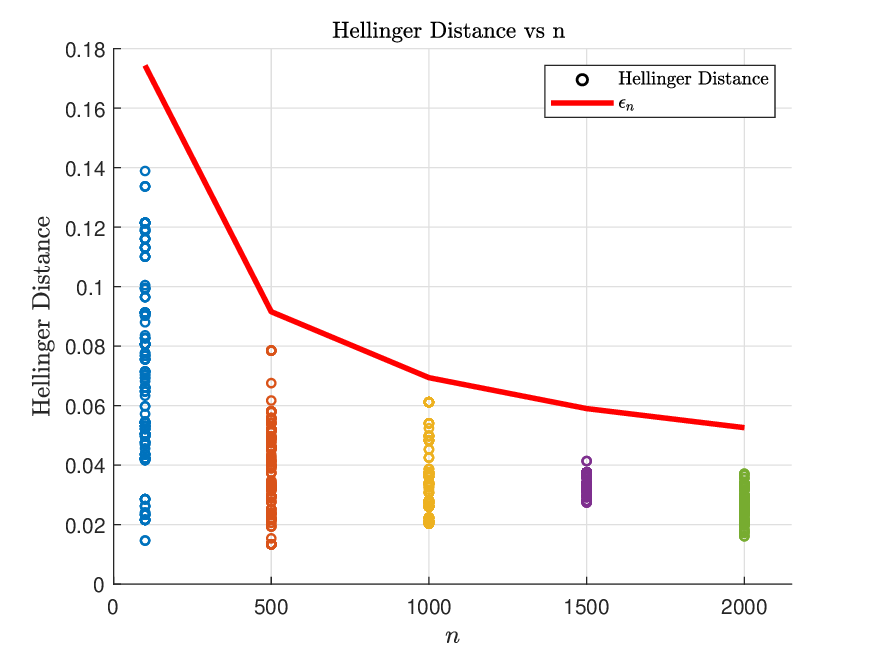}
  \caption{The $m=100$ Hellinger distances
with different sample sizes.}
  \label{fig_h}
\end{figure}

\subsection{Experiment 2: the Hellinger distance v.s. the sample size}

In this experiment, we shall numerically display the decreasing trend of the Hellinger distance as the sample size $n$ increases. Let us set $p=2$ as before and choose $n = [100,500, 1000, 1500, 2000]$, with corresponding $k_n\approx2\left(n^{\frac{6p+1}{7p(2p+1)}}\right)-1\approx[4,6,6,8,8]$. 
For each $n$, we plot $m=100$ Hellinger distances
$d_{H}(\hat{g},g_0)$ in Figure \ref{fig_h}. Each circle represents one computed Hellinger distance $d_{H}(\hat{g},g_0)$, while the red curve is of the order $\varepsilon_n=n^{-1/(2p+1)}=n^{-1/5}$.
It is observed that the Hellinger distances are generally below the red curve. Moreover, they decrease as $n$ increases, which is consistent with our theoretical result in Corollary \ref{cor1}. Further observation reveals that the Hellinger distance concentrates further as $n$ grows larger. This provides an additional numerical evidence that the estimator's performance improves with larger sample sizes $n$.

\section{Discussion}\label{sec-6}

In this paper, we have developed a unified Bayesian framework for density estimation based on weighted orthogonal polynomial expansions. By refining the analysis within the framework of \cite{shen2001}, we establish the minimax optimal posterior convergence rate regardless of whether the true density admits a strictly positive lower bound. Furthermore, we construct a Gaussian sieve prior that achieves this optimal rate.

However, our proposed sieve prior is non-adaptive, since its variance explicitly depends on the smoothness level of the true density. It is therefore highly desirable to develop adaptive priors that automatically adjust to unknown regularity \citep{Huang2004,AOS2007}. Accordingly, the construction of adaptive priors within our orthogonal polynomial framework constitutes an important direction for future research.
Another promising direction concerns density estimation for heavy-tailed distributions. To some extent, our results cover moderately heavy-tailed densities such as $g_0(x)\sim e^{-|x|^s}$ as $|x|\to\infty$, for any $0<s<2$, since orthogonal polynomial systems with respect to the weight $w(x)=e^{-|x|^s}$ can be constructed via the Gram-Schmidt procedure. Nevertheless, our theory cannot be extended to densities with polynomial tails, i.e., $g_0(x)\sim |x|^{-\alpha}$ as $|x|\to\infty$ for some $\alpha>1$, due to the absence of a valid orthogonal polynomial expansion. In particular, moments of order $k\geq\alpha$ fail to exist:
\begin{align}\label{eqn-6.2}\int_{\R^+} x^kw(x)dx=\infty.
\end{align}
This theoretical limitation is not unique in Bayesian nonparametrics. In Bayesian kernel mixture models, approximating heavy-tailed distributions using infinite mixtures of light-tailed kernels leads to posterior inconsistency in tail-index estimation \citep{Li2019}. A natural remedy is to use heavy-tailed kernels with polynomial decay, such as the inverse-gamma mixtures in \citep{EJOS2017}. However, such a modification is incompatible with our orthogonal polynomial framework, as noted earlier.

An alternative approach is to map the unbounded domain to a bounded interval. Along this line, a semiparametric method based on extreme value theory has been proposed \citep{Surya}, in which the cumulative distribution function of the generalized Pareto distribution maps heavy-tailed observations into the bounded interval $(0,1)$. One may incorporate such a mapping strategy into our orthogonal polynomial framework. Specifically, a well-designed heavy-tailed bijective transformation should be constructed to ensure that the Hellinger distances before and after transformation remain comparable. In this manner, the tail behavior of the true density is effectively decoupled from the central expansion, enabling the use of standard orthogonal polynomials for nonparametric density estimation.


\begin{funding}

The authors were supported by the National Natural Science Foundation of China (12271019) and the National Key R\&D Program of China (2022YFA1005103).
\end{funding}

\begin{supplement}
\stitle{Supplementary on ''The minimax optimal convergence rate of posterior density in the weighted orthogonal polynomials''}
\slink[url]{https://doi.org/10.1214/[provided by typesetter].pdf}
\sdescription{This supplementary PDF file contains four sections. 
Section A provides the detailed proof of Lemma \ref{lem-1}. 
Section B illustrates the specific specification of the sequence $\tilde{\gamma}_j$ using Legendre and Hermite polynomials. 
Section C contains the technical details in the proof of the sieve prior in Theorem \ref{th-sieve-gaussian}, including direct computations for the truncated prior mass and the tail probability bound.
Section D presents additional numerical experiments on unbounded intervals, evaluating the performance of the Bayesian estimator for an exponential density on $\R^+$ using Laguerre polynomials and a Gaussian density on $\R$ using Hermite polynomials.}
\end{supplement}

\bibliographystyle{ba}
\bibliography{reference}

\end{document}


\begin{frontmatter}
\title{Supplementary on ``The minimax optimal convergence rate of posterior density in the weighted orthogonal polynomials"}

\begin{aug}
\author[Y]{\fnms{Yiqi}~\snm{Luo}\ead[label=e1]{luoyiqi1999@163.com}}
\and
\author[Y,X,Z]{\fnms{Xue}~\snm{Luo}\ead[label=e2]{xluo@buaa.edu.cn}\orcid{0000-0003-1187-7599}}
\address[Z] {Corresponding author}
\address[Y]{School of Mathematical Sciences, Beihang University (Shahe Campus), Changping district, Beijing, 102206, P. R. China\printead[presep={,\ }]{e1}}

\address[X]{Key Laboratory of Mathematics, Informatics and Behavioral Semantics (LMIB), Haidian district, Beijing, 100191, P. R. China\printead[presep={,\ }]{e2}}

\runauthor{Y. Luo et al.}
\end{aug}




\end{frontmatter}

\section{The proof of Lemma 2.1}
\begin{lemma}[Lemma 2.1]
For any $\displaystyle g(x|\bbeta)=\sum_{j=0}^\infty\eta_jq_j(x)\in\G$ in (5), if $\bbeta\in\Omega$ in (6).
\end{lemma}

\begin{proof}
It is clear to see that the first two conditions (i.e. $\displaystyle\sum_{j=0}^
\infty\eta_jq_j(x)\geq0$ and \linebreak[4] $\displaystyle\sum_{j=0}^
\infty\eta_j\int_Iq_j(x)w(x)dx=1$) in the definition (6) of $\Omega$ correspond to $g\geq0$ and\linebreak[4] $\displaystyle\int_Ig(x)w(x)dx=1$, respectively. Thus, it is sufficient to show the last condition \linebreak[4]
$\displaystyle\sum_{j=0}^{\infty}\eta_j^2\tilde\gamma_j<\infty$ guarantees that $g\in W_{2,w}^p$.

For $l=0$, we have
\begin{equation}\label{g<}
\Vert g^{(0)}(x) \Vert_{L_{2,w}}^2 
    =  \int_I g^2(x) w(x)dx 
   =\int_I \left(\sum_{j=1}^{\infty} \eta_j q_j(x) \right)^2 w(x) dx
    \overset{(5)}=\sum_{j=0}^{\infty} \eta_j^2\gamma_j < \infty,
\end{equation}
by (6). 

For $1\leq l\leq p$, we have
\begin{align}\label{eqn-A.1}\notag
    \Vert g^{(l)}(x) \Vert_{L_{2,w}}^2&
\overset{(8)} =\int_I\left(\sum_{j=0}^{\infty} \eta_j \left( \sum_{i=0}^{j-1} a_{ij}^{(l)} q_i(x) \right)\right)^2w(x)dx\\\notag
    =&\int_I\left( \sum_{i=0}^{\infty} 
\left( \sum_{j=i+1}^{\infty} \eta_j a_{ij}^{(l)} \right) q_i(x)\right)^2w(x)dx
    =\sum_{i=0}^{\infty} \left(\sum_{j=i+1}^{\infty} \eta_j a_{ij}^{(l)} \right)^2 \gamma_i\\\notag
    =&\sum_{i=0}^{\infty} \left(\sum_{j=i+1}^{\infty} \eta_j a_{ij}^{(l)}\sqrt{ \gamma_i} \right)^2\\\notag
    \leq&\sum_{i=0}^{\infty} \left[\left( \sum_{j=i+1}^{\infty} \eta_j^2 (a_{ij}^{(l)})^4\gamma_i \right)^{1/2}
\left( \sum_{j=i+1}^{\infty} \eta_j^{2/3}\gamma_i^{1/3} \right)^{3/2}\right]\\
    \leq&
\left( \sum_{i=0}^{\infty} \sum_{j=i+1}^{\infty} 
\eta_j^2 (a_{ij}^{(l)})^4\gamma_i \right)^{1/2}
\left[ \sum_{i=0}^{\infty}
\left( \sum_{j=i+1}^{\infty} \eta_j^{2/3} \gamma_i^{1/3} \right)^3 
\right]^{1/2},
\end{align}
by applying H\"older's inequality twice. By exchanging the order of the double summation, the first term on the right-hand side of \eqref{eqn-A.1} equals to
\begin{equation}\label{eqn-A.2}
    \sum_{i=0}^{\infty} \sum_{j=i+1}^{\infty} 
\eta_j^2 (a_{ij}^{(l)})^4\gamma_i
    =\sum_{j=0}^{\infty}\eta_j^2 \sum_{i=0}^{j-1} (a_{ij}^{(l)})^4\gamma_i,
\end{equation}
while the second term is 
\begin{equation}\label{eqn-A.3}
    \sum_{i=0}^{\infty}
\left( \sum_{j=i+1}^{\infty} \eta_j^{2/3} \gamma_i^{1/3} \right)^3 
    \lesssim \sum_{j=1}^\infty j^4\eta_j^2\max_{0\leq i\leq j-1}\gamma_i
\end{equation}
by the weighted Hardy's inequality, with $a_j=\eta_j^{2/3}$ and $b_i=\gamma_i^{1/3}$ in Lemma \ref{lem-A.2} below. Combining \eqref{eqn-A.1}-\eqref{eqn-A.3}, one has
\begin{equation}\label{dg<}
    \Vert g^{(l)}(x) \Vert_{L_{2,w}}^2\lesssim\left(
    \sum_{j=0}^{\infty}\eta_j^2 \sum_{i=0}^{j-1} (a_{ij}^{(l)})^4\gamma_i\right)^{1/2}\left(
\sum_{j=1}^\infty j^4\eta_j^2\max_{0\leq i\leq j-1}\gamma_i\right)^{1/2}
    \leq\sum_{j=0}^\infty\eta_j^2\tilde\gamma_j,
\end{equation}
where $\displaystyle\tilde\gamma_j\geq\max\left\{\max_{0\leq l\leq \min(j,p)}\sum_{i=0}^{j-1} (a_{ij}^{(l)})^4\gamma_i,j^4\max_{0\leq i\leq j-1}\gamma_i\right\}$, for $j\geq0$.
\end{proof}

For the readers' convenience, we give the weighted Hardy's inequality and its proof below.
\begin{lemma}[Weighted Hardy's inequality]\label{lem-A.2}
    For two non-negative sequence $\{a_j\}_{j\geq1}$ and $\{b_i\}_{i\geq0}$, let $\displaystyle A_i:=\sum_{j=i+1}^\infty a_jb_i$, $i\geq0$, then 
    \[
        \sum_{i=0}^\infty A_i^3\lesssim\sum_{j=1}^\infty j^4a_j^3\max_{0\le i\le j-1}b_i^3.
    \]
\end{lemma}
\begin{proof}
By direct computation, one has
\[
\sum_{i=0}^{\infty} A_i^3 = \sum_{i=0}^{\infty}  \left( \sum_{j=i+1}^{\infty} a_j b_i \right)^3 
= \sum_{i=0}^{\infty} b_i^3 \left( \sum_{j=i+1}^{\infty} a_j \right)^3 ,
\]
where the inner summation can be estimated using the Cauchy-Schwarz inequality
\[
\sum_{j=i+1}^{\infty} a_j \leq \left( \sum_{j=i+1}^{\infty} j^3 a_j^3 \right)^{1/3} \left( \sum_{j=i+1}^{\infty} j^{-\frac{3}{2}} \right)^{2/3} \lesssim \left( \sum_{j=i+1}^{\infty} j^3 a_j^3 \right)^{1/3}.
\]
Therefore, 
\[
\sum_{i=0}^\infty 
A_i^3 \lesssim \sum_{i=0}^{\infty} b_i^3  \sum_{j=i+1}^{\infty} j^3 a_j^3 
    =\sum \limits_{j=1}^\infty j^3a_j^3 \sum \limits _{i=0}^{j-1}b_i^3 \le \sum_{j=1}^\infty j^4a_j^3\max_{0\le i\le j-1}b_i^3.
\]

\end{proof}

\section{Specification of $\Omega$ in two examples}

\begin{example}[Legendre polynomials] 
Let $q_j(x)=L_j(x)$ be the Legendre polynomial of degree $j$, $j\geq0$, orthogonal on $[-1,1]$ with respect to $w(x)\equiv 1$. The orthogonality relation is
\begin{equation*}
    \int_{-1}^1 L_i(x)L_j(x)dx=
\delta_{ij}\gamma_j,
\end{equation*}
where 
\begin{equation}\label{eqn-B.1}
    \gamma_j=\dfrac{2}{2j+1}\leq2,
\end{equation}
for all $j\geq0$. In this case, we have
\begin{equation*}
    \tilde\gamma_j\gtrsim j^{4p+1}
\end{equation*}
for $j\geq p$ in $\Omega$ (6).
\end{example}

We claim that $|a_{ij}^{(l)}|\lesssim j^l$, for all $0\leq i\leq j-1$, $j\geq1$ and $1\leq l\leq p$, by induction using the recurrence relation of Legendre polynomials \citep{lebedevNN}.

For $l=1$. Starting from $j=1$, $a_{01}^{(1)}\lesssim 1$ holds trivially. By induction, assume that $|a_{ij}^{(1)}|\lesssim j$ holds for all $0\leq i\leq j-1$, with $j=k-1$. Then
\begin{align}\label{eqn-A.4}\notag
\sum_{i=0}^{k}a_{i(k+1)}^{(1)}L_i(x)=&L_{k+1}'(x)
=L_{k-1}'(x)+(2k+1)L_k(x)\\
=&\sum_{i=0}^{k-2}a_{i(k-1)}^{(1)}L_i(x)+(2k+1)L_k(x),
\end{align}
where the second equality is the recurrence relation of Legendre polynomials and the third equality follows from the hypothesis assumption. Comparing both sides of \eqref{eqn-A.4}, $|a_{i(k+1)}^{(1)}|=|a_{i(k-1)}^{(1)}|\lesssim k-1$ for $0\leq i\leq k-1$ by the hypothesis assumption and $a_{k(k+1)}^{(1)}=2k+1\lesssim k+1$. Hence, $|a_{ij}^{(1)}|\lesssim j$ also holds for all $0\leq i\leq j-1$ with $j=k+1$.

For $l\geq2$. Again, by induction, assume that
$|a_{ij}^{(l-1)}|\lesssim j^{l-1}$ holds for all $0\leq i\leq j-(l-1)$ with $j=k$, and $|a_{ij}^{(l)}|\lesssim j^{l}$ holds for all $0\leq i\leq j-l$ with $j=k-1$.  Then
\begin{align}\label{eqn-A.5}\notag
   \sum_{i=0}^{k+1-l}a_{i(k+1)}^{(l)}L_i(x)=& L_{k+1}^{(l)}(x)
    =L_{k-1}^{(l)}(x)+(2k+1)L_k^{(l-1)}(x)\\\notag
=&\sum_{i=0}^{(k-1)-l}a_{i(k-1)}^{(l)}L_i(x)+(2k+1)\sum_{i=0}^{k-(l-1)}a_{ik}^{(l-1)}L_i(x)\\\notag
=&\sum_{i=0}^{(k-1)-l}\left(a_{i(k-1)}^{(l)}+(2k+1)a_{ik}^{(l-1)}\right)L_i(x)\\
    &+(2k+1)a_{(k-l)k}^{(l-1)}L_{k-l}(x)+(2k+1)a_{(k-l+1)k}^{(l-1)}L_{k-l+1}(x).
\end{align}
Comparing both sides of \eqref{eqn-A.5}, one has 
\begin{align*}
    |a_{i(k+1)}^{(l)}|=&|a_{i(k-1)}^{(l)}+(2k+1)a_{ik}^{(l-1)}|\lesssim (k-1)^l+(2k+1)k^{l-1},
\end{align*}
for all $0\leq i\leq k-l-1$, and
\begin{align*}
    |a_{(k-l)(k+1)}^{(l)}|=&|(2k+1)a_{(k-l)k}^{(l-1)}|\lesssim (2k+1)k^{l-1},\\
    |a_{(k-l+1)(k+1)}^{(l)}|=&|(2k+1)a_{(k-l+1)k}^{(l-1)}|\lesssim(2k+1)k^{l-1},
\end{align*}
which imply $|a_{i(k+1)}^{(l)}|\lesssim (k+1)^l$, for all $0\leq i\leq k-l+1$. Hence,
\begin{equation}\label{eqn-B.2}
    \sum_{i=0}^{j-1}\left(a_{ij}^{(l)}\right)^4
\lesssim j^{4l+1}.
\end{equation}
Therefore, by (7), $\displaystyle\tilde\gamma_j\overset{\eqref{eqn-B.1},\eqref{eqn-B.2}}\gtrsim\max\left\{\max_{0\leq l\leq \min(j,p)}j^{4l+1},j^4\right\}=j^{4p+1}$, for $j\geq p\geq1$.

\begin{example}[Hermite polynomials]
Let $q_j(x)=H_j(x)$ be the $j$-th Hermite polynomial with respect to $w(x)=e^{-x^2}$, so that
\begin{equation}
    \int_{\mathbb{R}}H_i(x)H_j(x)e^{-x^2}dx
=\delta_{ij}\gamma_j,
\end{equation}
where 
\begin{equation}\label{eqn-B.3}
    \gamma_j=2^j j!\sqrt{\pi}.
\end{equation}
In this case, $\tilde\gamma_j\gtrsim j^{4p}\gamma_j$, for $j\geq p\geq1$.
\end{example}

The Hermite polynomials satisfy the recurrence relation \citep{lebedevNN}
\begin{equation*}
H_j'(x)=2jH_{j-1}(x),
\end{equation*}
for $j\geq 1$, and
\begin{equation}
H_j^{(l)}(x)=2^l j(j-1)\cdots(j-l+1)H_{j-l}(x),
\end{equation}
for $0\leq l \leq j$. It is clear to see that
\[
    a_{(j-l)j}^{(l)}=2^l j(j-1)\cdots(j-l+1)\lesssim j^l,
\]
and $a_{ij}^{(l)}\equiv0$, for all  $0\leq i\leq j-l-1$. Hence, 
\begin{equation}\label{eqn-B.4}
    \sum_{i=0}^{j-l}(a_{ij}^{(l)})^4
\lesssim (j^{4l}).
\end{equation}
Therefore, 
\begin{equation*}
    \tilde\gamma_j\overset{\eqref{eqn-B.3},\eqref{eqn-B.4}}\gtrsim
\max\left\{\max_{0\leq l\leq\min(j,p)}j^{4l}2^jj!,j^42^jj!\right\}=j^{4p}2^jj!,
\end{equation*}
for $j\geq p\geq1$.

\section{Technical Details in the Proof of Theorem 4.1}

\subsection{Computation of $\pi_{<k_n}(B_{k_n})$}

Define $B_{k_n} = \left\{  \bbeta_n \Big| \sum \limits _{j=0}^{{k_n}-1} \left(\theta_{j,n} -\eta_{j,n} \right)^2\gamma_j <  \frac{1}{2}{a_n\varepsilon_n^2}\right\}$,
where $\bbeta_n=(\eta_{0,n},\eta_{1,n},\ldots)$:
\begin{align}\label{gau-betan}
  \eta_{0,n} &= 1+a_n-a_n\gamma_0, \notag\\
  \eta_{j,n} &\sim \N\!\left(0,\,(1-a_n\gamma_0)^2\,j^{-2p}\gamma_j^{-1}\right),
  \qquad \textup{independently for } 1\leq j\leq {k_n}-1, \notag\\
  \eta_{j,n} &= 0,\qquad j\geq {k_n}.
\end{align}
Then
\begin{align*}
\pi_{<{k_n}}(B_{k_n})
=& \int_{B_{k_n}}
\prod_{i=0}^{{k_n}-1}
\frac{i^p\gamma_i^{1/2}}{\sqrt{2\pi}\,(1-a_n\gamma_0)}
\exp\left\{
-\frac{\eta_{i,n}^2}
{2 i^{-2p}
\gamma_i^{-1} (1-a_n\gamma_0)^2}
\right\}
d\bbeta_{<k_n,n}
\\
\ge&
\int \limits _{\left\{  \bbeta_n \Big| \sum \limits _{j=0}^{{k_n}-1} \left(\theta_{j,n} -\eta_{j,n} \right)^2\gamma_j \le  \frac{1}{2}{a_n\varepsilon_n^2}\right\}}
\prod_{i=0}^{{k_n}-1}
\frac{i^p\,\gamma_i^{1/2}}{\sqrt{2\pi}}\\
&\phantom{\int \limits _{\left\{  \bbeta_n \Big| \sum \limits _{j=0}^{{k_n}-1} \left(\theta_{j,n} -\eta_{j,n} \right)^2\gamma_j \le  \frac{1}{2}{a_n\varepsilon_n^2}\right\}}}\cdot\exp\left\{
-\frac{\eta_{i,n}^2}
{2 i^{-2p}
\gamma_i^{-1} (1-a_n\gamma_0)^2}
\right\}
d\bbeta_{<k_n,n}.
\end{align*}
Let $w_{i,n} = \eta_{i,n} - \theta_{i,n},~i=0,\cdots,{k_n}-1$. Using the fact that $\frac{1}{2}\left
(w_{i,n}+\theta_{i,n}\right)^2 \le w_{i,n}^2+\theta_{i,n}^2$, we have
\begin{align*}
\pi_{<{k_n}}(B_{k_n})
&
\ge
\int \limits _{\left\{  \mathbf w_n \Big| \sum \limits _{j=0}^{{k_n}-1} w_{j,n}^2 \gamma_j \le  \frac{1}{2}{a_n\varepsilon_n^2}\right\}}
\prod_{i=0}^{{k_n}-1}
\frac{i^p\,\gamma_i^{1/2}}{\sqrt{2\pi}}
\exp\left\{
-\frac{(w_{i,n}+\theta_{i,n})^2}
{2 i^{-2p}
\gamma_i^{-1} (1-a_n\gamma_0)^2}
\right\}
\, d \mathbf{w}_{<k_n,n}
\\
&\ge
\prod_{i=0}^{{k_n}-1}
\exp\left\{
-\frac{\theta_{i,n}^2}
{ i^{-2p}
\gamma_i^{-1} (1-a_n\gamma_0)^2}
\right\}
\\
&\phantom{aaaa}\cdot
\int \limits _{\left\{  \mathbf w_n \Big| \sum \limits _{j=0}^{{k_n}-1} w_{j,n}^2 \gamma_j \le  \frac{1}{2}{a_n\varepsilon_n^2}\right\}}
\prod_{i=0}^{{k_n}-1}
\frac{i^p\,\gamma_i^{1/2}}{\sqrt{2\pi}}
\exp\left\{
-\frac{w_{i,n}^2}
{ i^{-2p}
\gamma_i^{-1} (1-a_n\gamma_0)^2}
\right\}
\,d \mathbf{w}_{<k_n,n}.
\end{align*}

Moreover,
since $\theta_n\in \Omega_n^{\G_{a_n}}$ (20), it follows that $\sum \limits _{i=0}^{k_n}\theta_{i,n}^2\tilde\gamma_i < n^{\frac1{2p+1}}$. Then 
\begin{align*}
\prod_{i=0}^{{k_n}-1}
\exp\left\{
-\frac{\theta_{i,n}^2}
{ i^{-2p}
\gamma_i^{-1} (1-a_n\gamma_0)^2}
\right\}
&=
\exp\left\{
- \sum_{i=0}^{{k_n}-1}
\frac{\theta_{i,n}^2}
{ i^{-2p}
\gamma_i^{-1} (1-a_n\gamma_0)^2}
\right\}
\\
&\ge
\exp\left\{
- 
\frac{1}{(1-a_n\gamma_0)^2}
\sum_{i=0}^{{k_n}-1}\theta_{i,n}^2 \tilde \gamma_i
\right\}\\
&\gtrsim \exp{\left\{ -\frac{1}{(1-a_n\gamma_0)^2}
n^{\frac{1}{2p+1}}
\right\}}.
\end{align*}
Let $u_{i,n}:=\sqrt{\gamma_i}w_{i,n},~i=0,\cdots,{k_n}-1$, so that 
\begin{align*}
\pi_{<{k_n}}(B_{k_n})
\gtrsim&
\exp{\left\{ -\frac{n^{{1}/{(2p+1)}}}{(1-a_n\gamma_0)^2}
\right\}}\\
&\phantom{a}\cdot\int _{\left\{  \mathbf w_n \Big| \sum \limits _{j=0}^{{k_n}-1} w_{j,n}^2 \gamma_j \le  \frac{1}{2}{a_n\varepsilon_n^2}\right\}}
\prod_{i=0}^{{k_n}-1}
\frac{i^p\gamma_i^{1/2}}{\sqrt{2\pi}}
\exp\left\{
-\frac{w_{i,n}^2}
{ i^{-2p}
\gamma_i^{-1} (1-a_n\gamma_0)^2}
\right\}
d \mathbf{w}_{<k_n,n}\\
=&\left(2\pi\right)^{\frac{{k_n}}{2}}
\left[ ({k_n}-1)!\right]^p
\exp{\left\{ -\frac{n^{{1}/{(2p+1)}}}{(1-a_n\gamma_0)^2}
\right\}}\\
&\phantom{a}\cdot\int _{\left\{  \mathbf u_n \Big| \sum \limits _{j=0}^{{k_n}-1} u_{j,n}^2 \le  \frac{1}{2}{a_n\varepsilon_n^2}\right\}}
\prod_{i=0}^{{k_n}-1}
\exp\left\{
-\frac{u_{i,n}^2}
{ i^{-2p}
(1-a_n\gamma_0)^2}
\right\}d\mathbf{u}_{<k_n,n}\\
:=&\left(2\pi\right)^{\frac{{k_n}}{2}}
\left[ ({k_n}-1)!\right]^p
\exp{\left\{ -\frac{n^{{1}/{(2p+1)}}}{(1-a_n\gamma_0)^2}
\right\}}I.
\end{align*}

The integral $I$ can be evaluated by applying Lemma 3 of \cite{shen2001}, which states that for any $r>0$ and any integrable function $f$,
\begin{align}\label{eqn-1}
\int_{\{\sum_{i=1}^{k_n} x_i^2 \le r^2\}}
f\!\left(\sqrt{\sum_{i=1}^{k_n} x_i^2}\right)
\, dx_1\cdots dx_{k_n}
=
\frac{\pi^{{k_n}/2}}{\Gamma({k_n}/2)}\,
r^{{k_n}}
\int_0^1
u^{{k_n}/2-1}
f(r\sqrt{u})\,du.
\end{align}
With
$f(x)=\exp\left\{-\frac{{k_n}^{2p}x^2}{(1-a_n\gamma_0)^2}\right\}$
and $r_n=\sqrt{\frac{1}{2}a_n\varepsilon_n^2}$ in \eqref{eqn-1}, we obtain
\begin{align*}
    I&\gtrsim
\int  _{\left\{  \mathbf u_n \Big| \sum \limits _{j=0}^{{k_n}-1} u_{j,n}^2 \le  \frac{1}{2}{a_n\varepsilon_n^2}\right\}}
\prod_{i=0}^{{k_n}-1}
\exp\left\{
-\frac{u_{i,n}^2}
{ {k_n}^{-2p}
(1-a_n\gamma_0)^2}
\right\}
\, d\mathbf{u}_{<k_n,n}\\
&=
\frac{\pi^{{k_n}/2}}{\Gamma({k_n}/2)}\, r_n^{{k_n}}
\int_0^1
u^{{k_n}/2-1}
\exp\left\{
-\frac{r_n^2 {k_n}^{2p} u}{(1-a_n\gamma_0)^2}
\right\}
\,du.
\end{align*}
For short of notations, let $v_n
=
\frac{a_n \varepsilon_n^2 {k_n}^{2p}}
{2(1-a_n \gamma_0)^2}u,
~
\Delta_n
=
\frac{a_n \varepsilon_n^2 {k_n}^{2p}}
{2(1-a_n \gamma_0)^2}$, then
\begin{align*}
    &\int_0^1
u^{\frac{{k_n}}{2}-1}
\exp\!\left\{
-\frac{a_n \varepsilon_n^2 {k_n}^{2p}}{2(1-a_n \gamma_0)^2} u
\right\}
du
    =\int_0^{\Delta_n}
\left( \frac{v_n}{\Delta_n} \right)^{\frac{{k_n}}{2}-1}
e^{-v_n}
\frac{1}{\Delta_n}
dv_n\\
    =&
\Delta_n^{-\frac{{k_n}}{2}}
\int_0^{\Delta_n}
v_n^{\frac{{k_n}}{2}-1} e^{-v_n}
dv_n
    \ge
\Delta_n^{-\frac{{k_n}}{2}}
e^{-\Delta_n}
\int_0^{\Delta_n}
v_n^{\frac{{k_n}}{2}-1}
dv_n
    =
\Delta_n^{-\frac{{k_n}}{2}}
e^{-\Delta_n}
\cdot
\frac{2}{{k_n}}\,
\Delta_n^{{k_n}/2}\\
    =&\frac{2}{{k_n}}\,
e^{-\Delta_n}.
\end{align*}
Note that  $a_n = \varepsilon_n^4=n^{-\frac{4p}{2p+1}}$ and ${k_n}=\O\left(n^{\frac{6p+1}{7p(2p+1)}}\right)$, then
\[
a_n \varepsilon_n^2 {k_n}^{2p}
=
\O\left(
n^{-\frac{6p(4p-1)}{(2p+1)(6p+1)}}
\right)
\to 0,
\qquad
a_n \gamma_0 \to 0.
\]
Hence,
\[
\exp{\{-\Delta_n\}}=\exp\!\left\{
-\frac{a_n \varepsilon_n^2 {k_n}^{2p}}
{2(1-a_n \gamma_0)^2}
\right\}
\gtrsim 1.
\]
Substituting this bound yields
\begin{align}\label{eqn-2}\notag
\pi_{<{k_n}}(B_{k_n})
&\gtrsim
\left(\frac12\right)^{k_n}
\left[({k_n}-1)!\right]^p
(a_n\varepsilon_n^2)^{{k_n}/2}
\frac{1}{\Gamma({k_n}/2)}
\frac{1}{{k_n}}
\exp{\left\{ -\frac{n^{{1}/{(2p+1)}}}{(1-a_n\gamma_0)^2}
\right\}}\\\notag
    &\gtrsim
\exp{\left\{ -\frac{n^{{1}/{(2p+1)}}}{(1-a_n\gamma_0)^2}
\right\}}
(a_n\varepsilon_n^2)^{\frac{{k_n}}{2}}
\left(\frac12\right)^{k_n}
\left(\frac{{k_n}}{2}\right)^{-\frac{{k_n}}{2}+\frac12}\\\notag
&
\qquad \cdot ~
e^{{k_n}/2}
({k_n}-1)^{p({k_n}-1)}
e^{-p({k_n}-1)}
\frac{1}{{k_n}}
\\
&\gtrsim
\exp{\left\{ -\frac{n^{{1}/{(2p+1)}}}{(1-a_n\gamma_0)^2}
\right\}}
(a_n\varepsilon_n^2)^{\frac{{k_n}}{2}}
\left(\frac12\right)^{\frac{3{k_n}}{2}}
{k_n}^{-\frac{{k_n}}{2}}
e^{-p{k_n}},
\end{align}
by 
\[
[({k_n}-1)!]^p
\ge
({k_n}-1)^{p({k_n}-1)}
\exp{\left\{-p({k_n}-1)\right\}},
\]
and Stirling’s approximation
\[
\Gamma\!\left(\frac{{k_n}}{2}\right)
=
\O\left(
\sqrt{2\pi}
\left(\frac{{k_n}}{2}\right)^{\frac{{k_n}}{2}-1}
\exp{\left\{-\frac{{k_n}}{2}\right\}}
\right).
\]

To estimate the lower bound of the right-hand side of \eqref{eqn-2} in the exponential form, we take the logarithm on both sides:
\begin{align*}
\log \pi_{<{k_n}}(B_{k_n})
&\gtrsim
-n^{\frac{1}{2p+1}}
+\frac{{k_n}}{2}\ln (a_n\varepsilon_n^2)
-\frac{3}{2}{k_n}\ln 2 
-\frac{{k_n}}{2}\ln {k_n}
- p{k_n} \\
&=-\left(
n^{\frac{1}{2p+1}}+
\frac{6p+1}{2}\,{k_n}\ln {k_n}
+\frac{3}{2}(\ln 2 + p)\,{k_n}\right)\\
&\gtrsim
-2n^{\frac{1}{2p+1}}\\
&=-2 a_n b_n \varepsilon_n^2,
\end{align*}
if we choose  $a_n = \varepsilon_n^4=n^{-\frac{4p}{2p+1}}$ and $b_n = \O\left(n^{1+\frac{4p}{2p+1}}\right)$.

\subsection{Computation of $\pi(\Omega_n^c)$}
Recall that $\Omega_n
:=\left\{\bbeta\in\Omega^{\G} \,\Bigg|\, \sum_{j=0}^{\infty}\eta_j^2\tilde\gamma_j
\le n^{\frac{1}{2p+1}}\right\}$.
Under the truncated Gaussian sieve $\pi_{<i}(\bbeta_{<i})=\prod_{j=1}^{i-1}\pi(\eta_j)$ of $\bbeta_{<i}:=(\eta_0,\cdots,\eta_{i-1})$, i.e.
\begin{align*}
  \eta_0 &= 1, \notag\\
  \eta_j &\sim \N\!\left(0,\,j^{-2p}\gamma_j^{-1}\right),
  \qquad \textup{for } 1\leq j\leq i-1, \notag\\
  \eta_j &= 0,\qquad j\geq i,
\end{align*}
for $i=2 \leq k_n$,
we have
\begin{align}\label{pinc}\notag
\pi(\Omega_n^c)
&= 1-\pi(\Omega_n)
\;\lesssim\;
1-\pi_{<2}(\Omega_n)  \\
&= 1-
\int_{\left\{
\eta_{1,n}\,\middle|\,
\sum\limits _{i=0}^1 \tilde\gamma_i \eta_{i,n}^2
\le n^{\frac{1}{2p+1}}
\right\}}
\frac{1}{\sqrt{2\pi}\,\gamma_1^{-1/2}}
\exp\!\left(
-\frac{\eta_{1,n}^2}
{2\gamma_1^{-1}}
\right)
\,d\eta_{1,n}.
\end{align}
Let us examine the integral domain in \eqref{pinc}. Since $\tilde\gamma_0=\gamma_0$ and $\eta_{0,n}=1$, then
\begin{align}\label{eqn-4}
    \left\{\eta_{1,n}\big|
    \tilde\gamma_0\eta_{0,n}^2
+\tilde\gamma_1\eta_{1,n}^2
\le n^{\frac{1}{2p+1}}
    \right\}
    \subset
    \left\{\eta_{1,n}\Big|
    \eta_{1,n}^2 \le 
    \frac{n^{1/(2p+1)-\gamma_0}}{\tilde \gamma_1}
    \right\}
    =
    \left\{\eta_{1,n}\Big|
    \left(\frac{\eta_{1,n}}{\sqrt{\gamma_1}}\right)^2 \le v^2
    \right\},
\end{align}
where $v:=\sqrt{\frac{n^{1/(2p+1)-\gamma_0}}{\gamma_1^{-1}\tilde \gamma_1}}=\O(n^{\frac{1}{2(2p+1)}})$.
Substituting \eqref{eqn-4} back to \eqref{pinc}, one has
\begin{align*}
        \pi(\Omega_n^c)& \lesssim1-\int_{-v}^v \frac{1}{\sqrt{2\pi}}\exp{\left\{-\frac{v^2}{2}\right\}}dv
        =1-2\left(\Phi(v)-1\right)\\
        &\lesssim 1- \Phi(v) 
        \lesssim
        \frac{2\exp{\left\{-v^2/2\right\}}}{v}
        \lesssim
        \frac{2\exp{\left\{-C_6n^{1/(2p+1)}\right\}}}{n}
        \lesssim 
         {\exp{\left\{-C_6n^{1/(2p+1)}\right\}}},
\end{align*}
where $\Phi(\cdot)$ is the distribution function of standard normal distribution. 
The third-to-last inequality follows from standard normal tail bound
\[
1-\Phi(x)
\le
\frac{2\exp(-x^2/2)}{x},
\qquad x>1.
\]

\section{Numerical experiments on unbounded interval}\label{ex2}

\begin{figure}[t]
  \centering
  \includegraphics[width=0.55\linewidth]{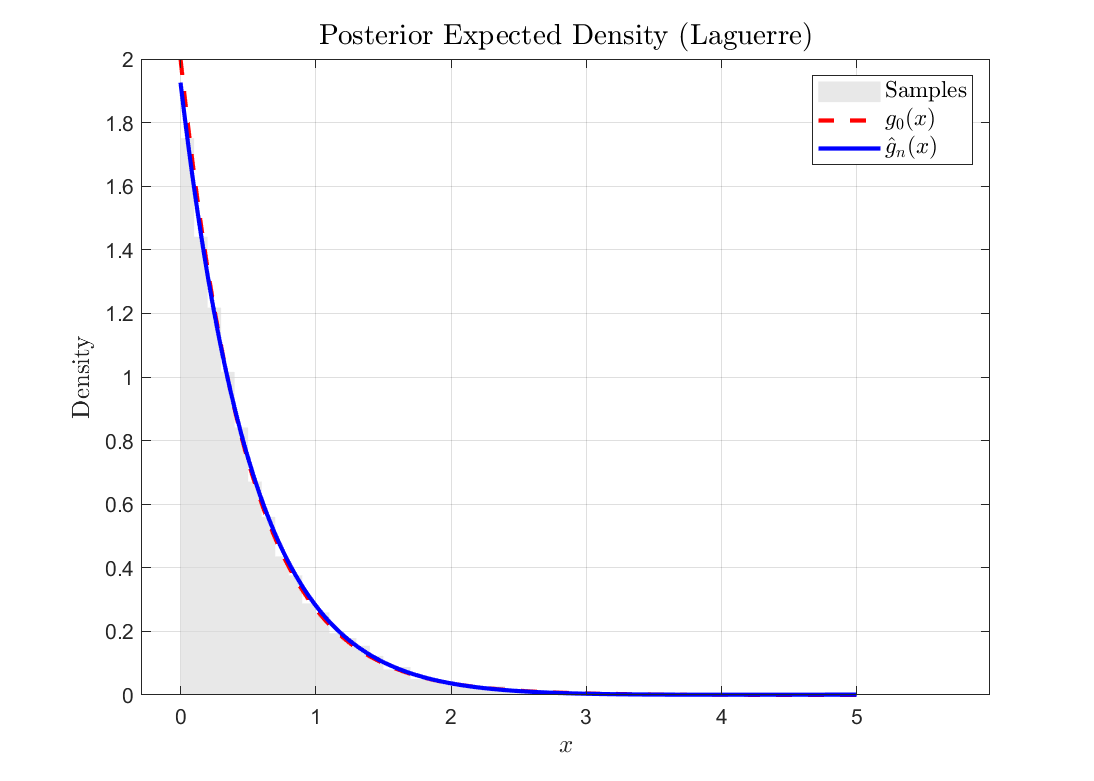}
  \caption{True density $g_0(\cdot)w(\cdot)$ (red dashed) and Bayesian estimator (blue) with $n=10{,}000$.}
  \label{fig-laguerre}
\end{figure}
\begin{figure}[t]
  \centering
  \includegraphics[width=0.55\linewidth]{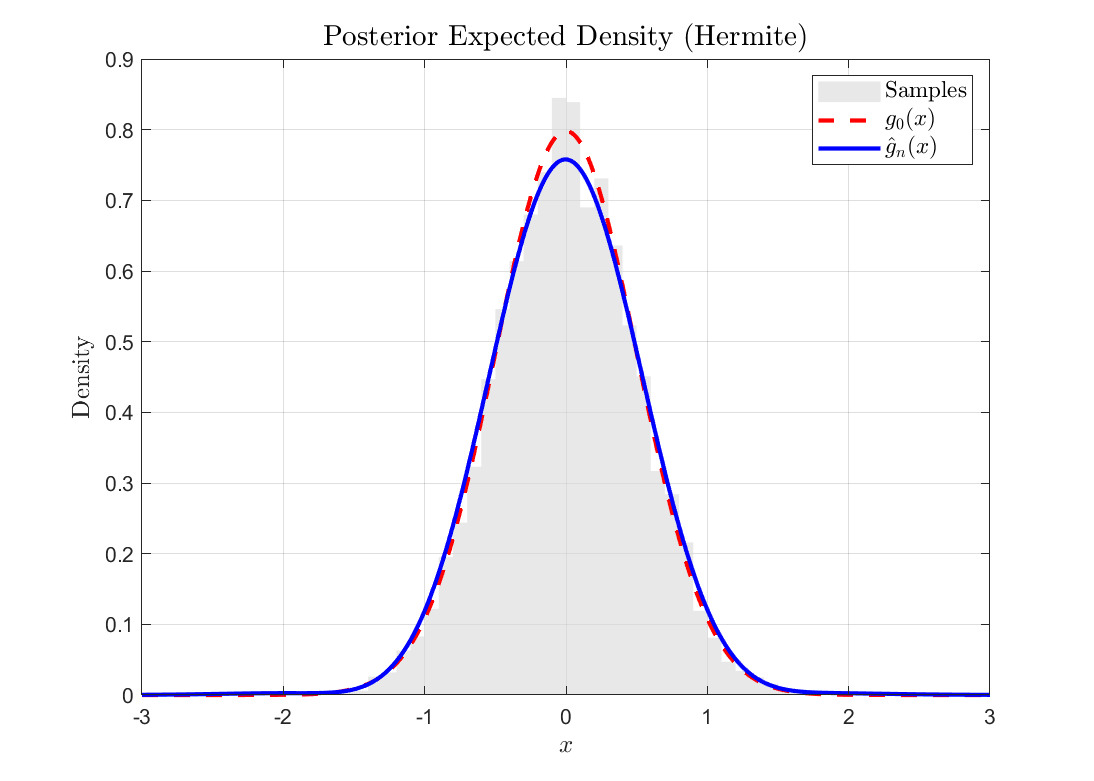}
  \caption{True density $g_0(\cdot)w(\cdot)$ (red dashed) and Bayesian estimator (blue) with $n=10{,}000$.}
  \label{fig-hermite}
\end{figure}

In this subsection, we shall consider two densities on the unbounded intervals $\R^+$ or $\R$, both of which have no strictly positive lower bounds.

\subsection{Exponential density on $\R^+$}\label{sec-D.1}
Let us consider the exponential density on $\R^+$ under the weight $w(x)=e^{-x}$:
\begin{equation}\label{truepdf-exp}
  g_0(x)=2e^{-x},
\end{equation}
that is, $\displaystyle\int_{\R^+}g_0(x)w(x)dx=\int_{\R^+}2e^{-2x}dx=1$. We shall use the first $10$ Laguerre polynomials $\{q_j(x)\}_{j=0}^{9}$ and the same truncated Gaussian sieve prior on $\bbeta$ as in (79) is imposed, i.e.
\begin{align}\label{eqn-5.8}
  \eta_j &\sim \N(0,\sigma_j^2), \qquad j=0,1,\ldots,9, \notag\\
  \eta_j &=0, \qquad j\geq 10.
\end{align}
As in Section 5.1, $n=10{,}000$ samples are drawn from $g_0(x)$, and
the Bayesian estimator is obtained by drawing Markov chain Monte Carlo (MCMC) samples from the posterior distribution induced by the prior \eqref{eqn-5.8}.
A random walk Metropolis algorithm is employed, generating $10000$ iterations where the first $2000$ are treated as burn-in.
The standard deviations are chosen as
\begin{equation}\label{eqn-5.9}
  \{\sigma_j\}_{j=0}^{9}
  =
  [1; 1; 0.25; 0.11; 0.06; 0.04; 0.03; 0.02; 0.02; 0.01].
\end{equation}

Figure \ref{fig-laguerre} shows that the estimator accurately captures the exponential shape on $\R^+$.

\subsection{Gaussian density on $\R$}\label{ex3}

Similar as in Section \ref{sec-D.1},  a Gaussian density under $w(x)=e^{-x^2}$
\begin{equation}\label{truepdf-gauss}
  g_0(x)=\sqrt{\frac{2}{\pi}}e^{-x^2},
\end{equation}
for $x\in\R$, is taken into consideration. The first $10$ Hermite polynomials are in use and Gaussian sieve prior on $\bbeta$ are imposed, with the  the standard deviations
\begin{equation}\label{eqn-5.11}
  \{\sigma_j\}_{j=0}^{9}
  =
  0.8\cdot[1; 0.53; 0.067; 0.012; 0.002; 0.001; 0; 0; 0; 0].
\end{equation}
Again, $n=10{,}000$ and the same MCMC algorithm as before, Figure \ref{fig-hermite} displays the estimator of the Gaussian \eqref{truepdf-gauss}.

\bibliographystyle{ba}
\bibliography{reference}